\documentclass[11pt, reqno]{amsart}
\usepackage{color}
\usepackage{amsmath}
\usepackage{latexsym,amssymb}
\usepackage[mathscr]{euscript}
\usepackage{stmaryrd}
\usepackage{bm}
\usepackage{enumerate}
\usepackage{setspace}
\newtheorem{thm}{Theorem}[section]
\newtheorem{proposition}{Proposition}[section]
\newtheorem{example}{\bf Example}[section]

\newtheorem{remark}{\bf Remark}[section]
\newtheorem{lemma}{\bf Lemma}[section]
\newtheorem{definition}{Definition}[section]
\numberwithin{equation}{section}
\newtheorem{assumption}{Assumption}[section]

\begin{document}
	
	\baselineskip=17pt
	
	\title[]
	{Discrete-time Zero-Sum Games for Markov chains with risk-sensitive average cost criterion}

\author[M. K. Ghosh]{Mrinal K. Ghosh}
\address{ Department of Mathematics\\
	Indian Institute of Science\\
	Bangalore-560012, India.}
\email{mkg@iisc.ac.in}

\author[S. Golui]{Subrata Golui}
\address{Department of Mathematics\\
	Indian Institute of Technology Guwahati\\
	Guwahati, Assam, India}
\email{golui@iitg.ac.in}

\author[C. Pal]{Chandan Pal}
\address{Department of Mathematics\\
	Indian Institute of Technology Guwahati\\
	Guwahati, Assam, India}
\email{cpal@iitg.ac.in}

\author[S. Pradhan]{Somnath Pradhan}
\address{Department of Mathematics and Statistics\\
	Queen's University\\
	Kingston, Ontario-K7L 3N6, Canada}
\email{sp165@queensu.ca}

	
	\date{}
	
	\begin{abstract}
		\vspace{2mm}
		\noindent
	 We study zero-sum stochastic games for controlled discrete time Markov chains with risk-sensitive average cost criterion with countable state space and Borel action spaces. The payoff function is nonnegative and possibly unbounded. Under a certain Lyapunov stability assumption on the dynamics, we establish the existence of a value and saddle point equilibrium. Further we completely characterize all possible saddle point strategies in the class of stationary Markov strategies.  Finally, we present and analyze an illustrative example.

		\vspace{2mm}
		
		
	\end{abstract}
	
	\maketitle
	
	\section{INTRODUCTION}
   We address  a risk-sensitive  discrete-time zero-sum game with long-run (or ergodic) cost criterion where the underlying state dynamics is given by a controlled Markov chains determined by a prescribed transition  kernel. The state space is a denumerable set, actions spaces are Borel spaces and the cost function is possibly
   unbounded. In \cite{BG} this problem is studied with bounded cost under a uniform ergodicity condition. Here we extend the results of \cite{BG}
    to the case with unbounded cost. This is carried out under a certain Lyapunov type stability condition.
      In the risk-neutral criterion, players consider the expected value of the total cost, but in the risk-sensitive criterion, they consider the expected value of the exponential of the total costs. As a result the risk-sensitive criterion  provides   comprehensive protection from  risk since it captures the effects of the higher order moments of the cost as well as its expectation; for more details see \cite{WH}.   We refer to \cite{GH4}, \cite{Z1} for risk-neutral Markov decision processes (MDP), and   \cite{GH1}, \cite{GH2}, \cite{WC4} for stochastic games with risk-neutral criterion.

 The analysis of stochastic systems  with the risk-sensitive average criteria can be traced back to the seminal papers by Jacobson in \cite{Jac} and Howard and Matheson in \cite{HM}.  The literature on risk-sensitive MDP under different cost criteria is quite extensive, e.g., \cite{BR1}, \cite{BP}, \cite{BG3}, \cite{CH}, \cite{GS2}, \cite{GL}, \cite{GZ1}, \cite{HEMA}, \cite{HM}, \cite{DMS1},  \cite{DMS3}, \cite{KP1}, \cite{KP2}, \cite{WC6}, \cite{WH}. The corresponding literature on discrete time ergodic risk-sensitive games can be found in \cite{BG}, \cite{BG1}, \cite{BR}, \cite{WC5}. The articles [\cite{BG}, \cite{BR}] address  zero-sum risk-sensitive stochastic games for discrete time Markov chains with discounted as well as ergodic criteria. Analogous results for continuous time Markov chains are carried out in  \cite{GKP}. The results of the article \cite{BG} are extended to the general state space in \cite{BR}. In \cite{BR},  the ergodic criterion is studied under a local minorization property and a Lyapunov condition. Nonzero-sum risk-sensitive average stochastic games for discrete time Markov chains studied in \cite{BG1}, \cite{WC5}. The papers \cite{BG}, \cite{BR}, \cite{GKP} treat the problem with bounded cost function but in many  real-life situations the cost functions  may be unbounded.

  The analysis of the ergodic cost criterion in \cite{BG}, \cite{BR} is carried out using vanishing discount asymptotics. Following \cite{BP} we study the ergodic cost criterion
  using an approach based on the principal eigenvalue associated with the corresponding Shapley equation.
  Under certain conditions we show that the risk-sensitive average cost optimality equation (Shapley equation) admits a principal eigenpair.
  We identify the principal eigenvalue as the value of the game. Then  we establish  the existence of a saddle-point equilibrium via the outer maximizing/minimizing selectors of the
  Hamiltonian associated with the Shapley equation. Additionally we give a complete characterization of all  saddle point strategies in the space of stationary Markov strategies.
  This will be explained in technical terms in the next section.


	The rest of this article is arranged as follows. Section 2 deals with problem description and preliminaries. In Section 3 we study Dirichlet eigenvalue problems. In Section 4, we show that the risk-sensitive optimality equation (i.e., Shapley equation) has a solution, obtain the value of the game and saddle-point equilibrium in the class of stationary Markov strategies. We also completely characterize all possible saddle point strategies in the class of stationary strategies in this section. In Section 5, we present an illustrative example. Section 6 concludes the paper with some remarks.
	
		\section{The game model}
	In this section we introduce a discrete-time zero-sum stochastic game model which consists of the following elements
	\begin{equation}
	\{{S}, {U},{V}, ({U}(i)\subset {U}, V(i)\subset V,i\in S), {P}( \cdot|i, u,v),{c}(i, u,v)
	\}. \label{eq 2.1}
	\end{equation}

  Here $S$ is the state space which is assumed to be the set of all nonnegative integers endowed with the discrete topology; ${U}$ and $V$ are  action spaces for players 1 and 2, respectively. The action spaces ${U}$ and $V$ are assumed to be Borel spaces with the Borel $\sigma$-algebras $\mathcal{B}({U})$ and $\mathcal{B}(V)$, respectively. For each $i\in S$, ${U}(i)\in \mathcal{B}({U})$ and $V(i)\in \mathcal{B}(V)$ denote the sets of admissible actions for players 1 and 2, respectively, when the systems is at state $i$. For any metric space $Y$, let $\mathcal{P}(Y)$ denote the space of probability measures on $\mathcal{B}(Y)$ with Prohorov topology. Next ${P}:\mathcal{K}\rightarrow \mathcal{P}(S)$ is a transition (stochastic) kernel, where $\mathcal{K}:=\{(i, u,v)|i\in S, u\in {U}(i), v\in V(i)\}$, a Borel subset of $S\times {U}\times V$\,. We assume that the function ${P}(j|i,u,v)$ is continuous in $(u,v)\in {U}(i)\times V(i)$ for any fixed $i,  j\in S$. Finally, the function ${c}:\mathcal{K} \to \mathbb{R}_{+}$ denotes the cost function which is assumed to be continuous in $(u,v)\in {U}(i)\times V(i)$\, for any fixed $i\in S$.\\
  The game evolves as follows. When the state $i\in S$ at time $t\in \mathbb{N}_0:=\{0,1,\cdots\}$,  players independently choose actions $u_t\in {U}(i)$ and  $v_t\in V(i)$ according to some strategies, respectively. As a consequence of this, the following happens:
	\begin{itemize}
		\item player 1 incurs an immediate cost  ${c}(i, u_t, v_t )$ and player 2 receives a reward  ${c}(i, u_t, v_t )$;
		\item the system moves to a new state $j\neq i$ with the probability determined by ${{P}(j|i,u_t,v_t)}$.
	\end{itemize}
	When the state of the system transits to a new state $j$, the above procedure repeats. Both the players have full information of  past and present states and past actions of both players.
	The goal of player 1 is to minimize his/her accumulated costs, whereas that of player 2 is to maximize the same with respect to some performance criterion $\mathscr{J}^{\cdot,\cdot}(\cdot,\cdot)$, which in our present case is defined by (\ref{eq 2.4}), below. At each stage, the players choose their actions on the basis of accumulated information. The available information for decision making at time $t\in \mathbb{N}_0$, i.e., the history of the process up to time $t$ is given by
	$$h_t:=(i^{'}_0,(u_0,v_0), i^{'}_1, (u_1,v_1), \cdots, i^{'}_{t-1},(u_{t-1},v_{t-1}),i^{'}_t),$$
	where $H_0=S$, $H_t=H_{t-1}\times ({U}\times V\times S),\cdots, H_\infty=(U\times V\times S)^\infty$ are the history spaces. An admissible strategy for player 1 is a sequence
	$\pi^1:=\{\pi^1_t:H_t\rightarrow {\mathcal{P}}({U})\}_{t\in \mathbb{N}_0}$ of stochastic kernels satisfying $\pi^1_t(U(X_t)|h_t) = 1$, for all $h_t\in H_t;\,\,t\geq 0$,\, where $X_t$ is the state process.
	The set of all such strategies for player 1 is denoted by $\Pi^1_{ad}$. A strategy for player 1 is called a Markov strategy i if
	$$\pi^1_t(\cdot|h_{t-1},u,v,i)=\pi^1_t(\cdot|h^{'}_{t-1},u^{'}, v^{'},i)$$
	for all $h_{t-1}, h^{'}_{t-1}\in H_{t-1}, u,u^{'}\in {U}, v,v^{'}\in V, i\in S, t\in \mathbb{N}_0$. Thus a Markov strategy for player 1 can be identified with a sequence of maps, denoted by $\pi^1\equiv\{\pi^1_t:S\rightarrow {\mathcal{P}}({U})\}_{t\in \mathbb{N}_0}$. A Markov strategy $\{\pi^1_t\}$ is called stationary Markov for player 1, if it does not have any explicit time dependence, i.e., $\pi^1_t(\cdot|h_t)=\tilde{\phi}(\cdot|i^{'}_t)$ for all $h_t\in H_t$ for some  mapping $\tilde{\phi}$ satisfying $\tilde{\phi}({U}(i)|i)=1$ for all $i\in S$. The set of all Markov strategies and all stationary Markov strategies for player 1, are denoted by $\Pi^1_{M}$ and $\Pi^1_{SM}$, respectively. Similarly, the set of all admissible strategies, Markov strategies and stationary Markov strategies for player 2 are defined similarly and denoted by $\Pi^2_{ad}$, $\Pi^2_M$, and $\Pi^2_{SM}$, respectively.
	 For each $i ,j\in  S$, $\mu\in \mathcal{P}({U}(i))$ and $\nu\in \mathcal{P}(V(i))$, the cost function $c$ and the transition kernel $P$ are extended as follows:
$$c(i,\mu,\nu):=\int_{V(i)}\int_{U(i)}{c}(i,u,v)\mu(du)\nu(dv),$$
	$$P(j|i,\mu,\nu):=\int_{V(i)}\int_{U(i)}{P}(j|i,u,v)\mu(du)\nu(dv)$$
(by an abuse of notation we use the same notation $c$ and $P$).
	For a given initial distribution $\tilde{\pi}_0\in \mathcal{P}(S)$ and a pair of strategies $(\pi^1,\pi^2)\in \Pi^1_{ad}\times \Pi^2_{ad}$, by Tulcea's Theorem (see Proposition 7.28 of \cite{BS}), there exists unique probability measure ${P}^{\pi^1,\pi^2}_{\tilde{\pi}_0}$ on $(\Omega,\mathscr{B}(\Omega))$, where $\Omega=(S\times {U}\times V)^\infty$. When $\tilde{\pi}_0=\delta_i$, $i\in S$ this probability measure is simply written by ${P}^{\pi^1,\pi^2}_{i}$ satisfying
	\interdisplaylinepenalty=0
	\begin{align}
{P}^{\pi^1,\pi^2}_i(X_0=i)=1~\text{ and}~	{P}^{\pi^1,\pi^2}_i(X_{t+1}\in {A}|H_t,\pi^1_t,\pi^2_t)=P({A}|X_t,\pi^1_t,\pi^2_t)~\forall~A\in \mathscr{B}(S).\label{eq 2.2}
	\end{align}

Let $E_{i}^{\pi^1,\pi^2} $ denote the expectation with respect to the probability measure ${P}^{\pi^1,\pi^2}_i$. Now from [\cite{H}, p. 6], we know that under any  $(\pi^1,\pi^2)\in \Pi^1_M\times \Pi^2_M$, the corresponding stochastic process $X$ is strong Markov.\\
We now introduce some notations.\\
\textbf{Notations:}\\
For any finite set $\tilde{\mathscr{D}}\subset S$, we define $\mathcal{B}_{\tilde{\mathscr{D}}} = \{f:S\to\mathbb{R}\mid \,\, f(i) = 0\,\,\, \forall \,\, i\in {\tilde{\mathscr{D}}}^c\}$,\;
$\mathcal{B}_{\tilde{\mathscr{D}}}^{+}\subset \mathcal{B}_{\tilde{\mathscr{D}}}$ denotes the cone of all nonnegative functions vanishing outside $\tilde{\mathscr{D}}.$
Given any real-valued function ${\mathcal{V}}\geq 1$ on $S$, we define a Banach space $(L^\infty_{{\mathcal{V}}},\|\cdot\|^\infty_{\mathcal{V}})$ of ${\mathcal{V}}$-weighted  functions by
$$L^\infty_{\mathcal{V}}=\biggl\{f: S\rightarrow\mathbb{R}\mid \|f\|^\infty_{\mathcal{V}}:=\sup_{i\in  S}\frac{|f(i)|}{{\mathcal{V}}(i)}< \infty\biggr\}.$$
For any ordered Banach space $\tilde{\mathcal{X}}$, a subset $\tilde{\mathcal{C}}\subset \tilde{\mathcal{X}}$ and $x,y\in \tilde{\mathcal{X}}$, we define $\succeq$ as $x\succeq y\Leftrightarrow x-y\in \tilde{\mathcal{C}}$, i.e., the partial ordering in $\tilde{\mathcal{X}}$ with respect to the cone $\tilde{\mathcal{C}}$. For any subset $\hat{\mathscr{B}}\subset S$, $\check{\tau}(\hat{\mathscr{B}})=\inf\{t:X_t\in \hat{\mathscr{B}}\}$, i.e., the first entry time of $X_t$ to $\hat{\mathscr{B}}$. Also, for any subset $\tilde{\mathscr{D}}\subset S$, $\tau({\tilde{\mathscr{D}}}):=\inf\{t>0:X_t\notin {\tilde{\mathscr{D}}}\}$ denotes the first exit time from ${\tilde{\mathscr{D}}}$.\\
We now introduce the cost criterion.\\
\textbf{Ergodic cost criterion:}
Now we define the risk-sensitive average cost criterion for zero-sum discrete-time games. Let $\theta > 0$ be the risk-sensitive parameter. For each $i\in S$ and any $(\pi^1,\pi^2) \in \Pi^1_{ad}\times \Pi^2_{ad}$, the risk-sensitive ergodic cost criterion is given by
\begin{equation}
\mathscr{J}^{\pi^1,\pi^2}(i,c):= \limsup_{T\rightarrow \infty}\frac{1}{ T}\ln E^{\pi^1,\pi^2}_i\biggl[e^{\theta \sum_{t=0}^{T-1}\ c(X_t,\pi^1_t,\pi^2_t)}\biggr].\label{eq 2.4}
\end{equation}
Since the risk-sensitive parameter remains the same throughout, we assume without loss of generality that $\theta = 1$.
The lower value and upper value of the game, are  functions on $S$, defined as \\
$\displaystyle\mathscr{L} (i):=\sup_{\pi^2\in \Pi^2_{ad}}\inf_{\pi^1\in \Pi^1_{ad}}\mathscr{J}^{\pi^1,\pi^2}(i,c)$ and $\displaystyle \mathscr{U}(i):=\inf_{\pi^1\in \Pi^1_{ad}}\sup_{\pi^2\in \Pi^2_{ad}}\mathscr{J}^{\pi^1,\pi^2}(i,c)$ respectively. It is easy to see that $$\mathscr{L}(i)\leq\mathscr{U}(i)~ \text{for all}~i\in S.$$ If $\mathscr{L}(i)=\mathscr{U}(i)$ for all $i\in S$, then the common function is called the value of the game and is denoted by $\mathscr{J}^{*}(i)$. A strategy $\pi^{*1}$ in $\Pi^1_{ad}$ is said to be optimal for player 1 if
	$$\mathscr{J}^{\pi^{*1},\pi^{2}}(i,c)\leq \sup_{\pi^2\in \Pi^2_{ad}}\inf_{\pi^1\in \Pi^1_{ad}}\mathscr{J}^{\pi^1,\pi^2}(i)=\mathscr{L}(i)~ \forall~ i\in S,~ \forall \pi^2\in \Pi^2_{ad}.$$
	Similarly, $\pi^{*2}\in \Pi^2_{ad}$ is optimal for player 2 if
	$$\mathscr{J}^{\pi^{1},\pi^{*2}}(i,c)\geq \inf_{\pi^1\in \Pi^1_{ad}}\sup_{\pi^2\in \Pi^2_{ad}}\mathscr{J}^{\pi^1,\pi^2}(i)=\mathscr{U}(i) ~ \forall~ i\in S,~ \forall \pi^1\in \Pi^1_{ad}.$$
	If $\pi^{*k}\in \Pi^k_{ad}$ is optimal for player k (k=1,2), then $(\pi^{*1},\pi^{*2})$ is called a pair of optimal strategies. The pair of strategies $(\pi^{*1},\pi^{*2})$ at which this value is attained i.e., if
	 $$	\mathscr{J}^{{\pi}^{*1},\pi^{2}}(i,c) \leq  \mathscr{J}^{{\pi}^{*1},{\pi}^{*2}}(i,c) \leq \mathscr{J}^{\pi^1,{\pi}^{*2}}(i,c),\,~\forall \pi^1\in \Pi^1_{ad},~\forall \pi^2\in \Pi^2_{ad},$$
	then the pair $(\pi^{*1},\pi^{*2})$ is called a saddle-point equilibrium, and then $\pi^{*1}$ and $\pi^{*2}$ are optimal for player 1 and player 2, respectively.\\
Following \cite{BG}, the Shapley equation equation for the above problem is given by
\begin{align}
e^{\rho}\psi(i)&=\sup_{\nu\in \mathcal{P}(V(i))}\inf_{\mu\in \mathcal{P}(U(i))}\bigg[e^{c(i,\mu,\nu)}\sum_{j\in S}\psi(j) P(j|i,\mu,\nu)\bigg]\nonumber\\
&=\inf_{\mu\in \mathcal{P}(U(i))}\sup_{\nu\in \mathcal{P}(V(i))}\bigg[e^{c(i,\mu,\nu)}\sum_{j\in S}\psi(j)P(j|i,\mu,\nu)\bigg], ~i\in S \nonumber.
\end{align}
In the above equation, $\rho$ is a scalar and $\psi$ is an appropriate function.

Our goal is to establish the existence of a saddle-point equilibrium among the class of admissible history-dependent strategies and provide its complete characterization.
We now describe briefly our technique for establishing the existence of a saddle-point equilibrium. We first construct an increasing sequence of bounded subsets of the state space $S$. Then we apply Krein-Rutman theorem \cite{A1} on each bounded subset to obtain a bounded solution of the corresponding  Dirichlet eigenvalue problem, i.e., a solution to the above equation on each finite subset with the condition that the solution is zero in the complement of that subset . Using a suitable Lyapunov stability condition
(to be stated shortly), we pass to the limit and show that risk-sensitive zero sum ergodic optimality equation admits a principal eigenpair. Subsequently we establish a stochastic representation of the principal eigenfunction. This enables us to characterize all possible saddle point equilibria in the space of stationary Markov strategies. To this end we make certain assumptions.
First we define a norm-like function which is used in our assumptions.
\begin{definition}
		A function $f:S\rightarrow\mathbb{R}$ is said to be norm-like if for every $k\in \mathbb{R}$, the set $\{i:f(i)\leq k\}$ is either empty or finite.
	\end{definition}
 Since the  cost function (i.e., ${c}(i,  u,v)$ ) may be unbounded, to guarantee the finiteness of $\mathscr{J}^{ \pi^1,\pi^2}(i,c)$, we use the following assumption.
\begin{assumption}\label{assm 2.2}
	We assume that the Markov chain $\{X_t\}_{t\geq 0}$ is irreducible under every  pair of stationary Markov strategies $(\pi^1,\pi^2)\in \Pi_{SM}^1 \times \Pi_{SM}^2$.
Also, suppose there exist a constant $\tilde{C}>0$, a real-valued function $\mathcal{W}\geq 1  $ on $S$ and, a finite set $\tilde{\mathscr{K}}$ such that one of the following hold.
	\begin{itemize}
		\item[(a)]\textbf{If the running cost is bounded:} For some positive constant $\tilde{\gamma} > \|c\|_{\infty}$, we have the following blanket stability condition
		\begin{align}
	\sup_{(u,v)\in U(i)\times V(i)}\sum_{j\in S}\mathcal{W}(j){P}(j|i,u,v)\leq \tilde{C} I_{\tilde{\mathscr{K}}}(i)+e^{-\tilde{\gamma}} \mathcal{W}(i) ~\forall i\in S,\label{eq 2.5}
		\end{align}
	where $\|{c}\|_\infty:=\displaystyle \sup_{(i,u,v)\in \mathcal{K}}{c}(i,u,v)$.
		\item[(b)]\textbf{If the running cost is unbounded:} For some real-valued nonnegative norm-like function $\tilde{\ell}$ on $S$ it holds that
		\begin{align}
	\sup_{(u,v)\in U(i)\times V(i)}\sum_{j\in S}\mathcal{W}(j){P}(j|i,u,v)\leq \tilde{C} I_{\tilde{\mathscr{K}}}(i)+e^{-\tilde{\ell}(i)}\mathcal{W}(i) ~\forall i\in S,\label{eq 2.6}
		\end{align}
		where the function $\displaystyle  \tilde{\ell}(\cdot)-\max_{(u,v)\in U(\cdot)\times V(\cdot)}{c}(\cdot,u,v)\;$ is norm-like.
	\end{itemize}
\end{assumption}
The condition (\ref{eq 2.6}) is useful to treat problems with unbounded running cost. We prove that, (\ref{eq 2.6}) implies (\ref{eq 2.4}) is finite. Similar condition is also used in [\cite{BM}, Theorem 1.2],  [\cite{BG3}, Theorem 2.2] in the study of multiplicative ergodicity. Also, see \cite{WC5}.

We wish to establish the existence of a saddle point equilibrium in the space of stationary Markov strategies. For the existence of saddle point equilibrium, we make the following assumptions.
	\begin{assumption}\label{assm 2.3}
	\begin{enumerate}	
		\item [(i)] The admissible action spaces $U(i)(\subset U)$ and $V(i)(\subset V)$ are compact for each $i\in S$.
		
		\item [(ii)] There exists $i_0\in S$ such that ${P}(j|i_0,u,v)>0$ for all $j\neq i_0$ and $(u,v)\in U(i_0)\times V(i_0)$.
	
		\item [(iii)]  $(i,u,v)\rightarrow\displaystyle \sum_{j\in S}{\mathcal{W}}(j){P}(j|i,u,v)$ is continuous in $(u,v)\in U(i)\times V(i)$, where $\mathcal{W}$ is the function defined in Assumption~\ref{assm 2.2}\,. 		
			\end{enumerate}	
\end{assumption}
\begin{remark}
(1) Assumption \ref{assm 2.3} (i) and (iii) are standard continuity-compactness assumption.

(2)  Under Assumption \ref{assm 2.3} (i), for each $i\in S$, by Proposition 7.22 in [\cite{BS} p. 130], we know that $ \mathcal{P}({U}(i))$ and $ \mathcal{P}(V(i))$ are compact and metrizable. Note that $\pi^1 \in {\Pi_{SM}^1}$ can be identified with a map $\pi^1: S \to {\mathcal{P}}({U})$ such that  $ \pi^1(\cdot|i) \in {\mathcal{P}}({U}(i))$ for each $i \in S$. Thus, we have $\displaystyle \Pi_{SM}^1=\Pi_{i\in S} {\mathcal{P}}({U}(i))$. Similarly, $\displaystyle \Pi_{SM}^2=\Pi_{i\in S} {\mathcal{P}}(V(i))$.  Therefore by Tychonoff theorem, the sets ${\Pi_{SM}^1}$ and ${\Pi_{SM}^2}$ are compact metric spaces endowed with the  product topology. Also, it is clear that these spaces are locally convex topological vector spaces.

(3)  Assumption \ref{assm 2.3} (ii) will be used to show that the limit of the sequence of Dirichlet eigenfunctions does not vanish in the limit (see Lemma \ref{L2.4} below). It is also possible to consider other type of condition instead of Assumption \ref{assm 2.3} (ii). We refer to Remark \ref{R1} for further discussion.
\end{remark}

Using generalized Fatou's lemma as in [\cite{HL}, Lemma 8.3.7], from Assumption~\ref{assm 2.3} one can easily get the following result, which will be used in subsequent sections; we omit the details.
\begin{lemma}\label{lemm 3.1}
	Under Assumptions \ref{assm 2.2} and \ref{assm 2.3}, the functions
	$\sum_{j\in S}P(j|i,\mu,\nu)f(j)$ and $c(i,\mu,\nu)$  are continuous  at $(\mu,\nu)$ on $\mathcal{P}(U(i))\times \mathcal{P}(V(i))$ for each fixed $f \in L_{\mathcal{W}}^{\infty}$ and $i\in S$.
\end{lemma}

\section{Dirichlet eigenvalue problems}
We begin this section by stating
 a version of the nonlinear Krein-Rutman theorem from [\cite{A1}, Section 3.1] which plays a crucial role in our analysis of the Dirichlet eigenvalue problems.
\begin{thm}\label{theo 2.1}
	Let $\tilde{\mathcal{X}}$ be an ordered Banach space and $\tilde{\mathcal{C}} $  a nonempty closed subset of $\tilde{\mathcal{X}}$ satisfying $\tilde{\mathcal{X}}=\tilde{\mathcal{C}}-\tilde{\mathcal{C}}$. Let $\tilde{T}:\tilde{\mathcal{X}}\rightarrow\tilde{\mathcal{X}}$ be a 1-homogeneous, order-preserving, continuous, and compact map satisfying the property that for some nonzero $\zeta\in \tilde{\mathcal{C}}$ and $\hat{N}>0$, we have $\hat{N}\tilde{T}(\zeta)\succeq \zeta$. Then there exists a nontrivial $\hat{f}\in \tilde{\mathcal{C}}$ and a scalar $\tilde{\lambda}>0$, such that $\tilde{T}\hat{f}=\tilde{\lambda}\hat{f}$.
\end{thm}

 Let $i_0\in S$ be the reference state as in Assumption 2.2 (ii). Consider an increasing sequence of finite subsets $\tilde{\mathscr{D}}_n\subset S$ such that $\cup_{n=1}^{\infty}\tilde{\mathscr{D}}_n=S$ and $i_0\in \tilde{\mathscr{D}}_n$ for all $n\in \mathbb{N}$. The following lemma plays a crucial role in our analysis.
\begin{lemma}\label{lemm 2.4}
Suppose that Assumption \ref{assm 2.2} holds. Let $\tilde{\mathscr{B}}\supset \tilde{\mathscr{K}}$ be a finite subset of $S$ and let $\check{\tau}(\tilde{\mathscr{B}})=\inf\{t:X_t\in \tilde{\mathscr{B}}\}$, be the first entry time of $X_t$ to $\tilde{\mathscr{B}}$. Then for any pair of strategies $(\pi^1,\pi^2)\in \Pi^1_{ad}\times \Pi^2_{ad}$ we have the following:
	\begin{enumerate}
		\item [(i)] If Assumption \ref{assm 2.2} (a) holds: Then
		\begin{align}
		E^{\pi^1,\pi^2}_i\bigg[e^{\tilde{\gamma} \check{\tau}(\tilde{\mathscr{B}})}\mathcal{W}(X_{\check{\tau}(\tilde{\mathscr{B}})})\bigg]\leq \mathcal{W}(i)~ \forall~i\in {\tilde{\mathscr{B}}}^c.\label{eq 2.21}
		\end{align}
		\item [(ii)] If Assumption \ref{assm 2.2} (b) holds:
		\begin{align}
		E^{\pi^1,\pi^2}_i\bigg[e^{\sum_{s=0}^{\check{\tau}(\tilde{\mathscr{B}})-1}\tilde{\ell}(X_s) }\mathcal{W}(X_{\check{\tau}(\tilde{\mathscr{B}})})\bigg]\leq \mathcal{W}(i)~\forall~i\in {\tilde{\mathscr{B}}}^c.\label{eq 2.22}
		\end{align}
			\end{enumerate}
		\end{lemma}
	\begin{proof}
This result is proved in Lemma 2.3 in \cite{BP} for  one controller case. The proof for two controller case is analogous.

		\end{proof}
%
%
Now we prove the following existence result which is useful in establishing the existence of a Dirichlet eigenpair.
\begin{proposition}\label{prop 2.1}
	Suppose Assumption \ref{assm 2.3} holds.
Take any function $\bar{c}:\mathscr{K}\rightarrow \mathbb{R}$ which is continuous in $(u,v)\in U(i)\times V(i)$ for each fixed $i\in S$ satisfying the relation $\bar{c}<-\delta$ in $\tilde{\mathscr{D}}_n$, where $\delta>0$ is a constant and $\tilde{\mathscr{D}}_n$ is a finite set as described previously. Then for any $g\in \mathcal{B}_{\tilde{\mathscr{D}}_n}$, there exits a unique solution $\varphi\in \mathcal{B}_{\tilde{\mathscr{D}}_n}$ to the following nonlinear equation
\begin{align}
\varphi(i)&=\inf_{\mu\in \mathcal{P}(U(i))}\sup_{\nu\in \mathcal{P}(V(i))}\biggl[e^{\bar{c}(i,\mu,\nu)}\sum_{j\in S}\varphi(j)P(j|i,\mu,\nu)+g(i)\biggr]\nonumber\\
&=\sup_{\nu\in \mathcal{P}(V(i))}\inf_{\mu\in \mathcal{P}(U(i))}\biggl[e^{\bar{c}(i,\mu,\nu)}\sum_{j\in S}\varphi(j)P(j|i,\mu,\nu)+g(i)\biggr]~\forall i\in \tilde{\mathscr{D}}_n.\nonumber\\
\label{eq 2.7}
\end{align}
Moreover, we have
\interdisplaylinepenalty=0
\begin{align}
\varphi(i)&=\inf_{\pi^1\in \Pi^1_{ad}}\sup_{\pi^2\in \Pi^2_{ad}}E^{\pi^1,\pi^2}_i\biggl[\sum_{t=0}^{\tau(\tilde{\mathscr{D}}_n)-1}e^{\sum_{s=0}^{t-1}\bar{c}(X_s,\pi^1_s,\pi^2_s)}g(X_t)\biggr]\nonumber\\
&=\sup_{\pi^2\in \Pi^2_{ad}}\inf_{\pi^1\in \Pi^1_{ad}}E^{\pi^1,\pi^2}_i\biggl[\sum_{t=0}^{\tau(\tilde{\mathscr{D}}_n)-1}e^{\sum_{s=0}^{t-1}\bar{c}(X_s,\pi^1_s,\pi^2_s)}g(X_t)\biggr]~\forall i\in S,\label{eq 2.8}
\end{align} where $\tau(\tilde{\mathscr{D}}_n):=\inf\{t>0:X_t\notin \tilde{\mathscr{D}}_n\}$, first exit time from $\tilde{\mathscr{D}}_n$.
\end{proposition}
\begin{proof}
Let $g\in \mathcal{B}_{\tilde{\mathscr{D}}_n}$.
		Define a map  $\hat{T}:\mathcal{B}_{\tilde{\mathscr{D}}_n}\rightarrow\mathcal{B}_{\tilde{\mathscr{D}}_n}$ by
		\begin{align}
	&	\sup_{\nu\in \mathcal{P}(V(i))}\inf_{\mu\in \mathcal{P}(U(i))}\biggl[e^{\bar{c}(i,\mu,\nu)}\sum_{j\in S}\tilde{\phi}(j){P}(j|i,\mu,\nu)+g(i)\biggr]=\hat{T}\tilde{\phi}(i),~i\in\tilde{\mathscr{D}}_n, \tilde{\phi}\in \mathcal{B}_{\tilde{\mathscr{D}}_n}\nonumber\\
	&~~~~~~~\text{ and}~\hat{T}\tilde{\phi}(i)=0~\text{ for }~i\in \tilde{\mathscr{D}}_n^c.\label{eq 2.10}
		\end{align}
		Now, let $\tilde{\phi}_1, \tilde{\phi}_2\in \mathcal{B}_{\tilde{\mathscr{D}}_n}$. Then
		\begin{align*}
		(\hat{T}\tilde{\phi}_2(i)-\hat{T}\tilde{\phi}_1(i))\leq \max_{i\in \tilde{\mathscr{D}}_n}	\sup_{\nu\in \mathcal{P}(V(i))}\sup_{\mu\in \mathcal{P}(U(i))}e^{\bar{c}(i,\mu,\nu)}\|\tilde{\phi}_2-\tilde{\phi}_1\|_{\tilde{\mathscr{D}}_n}.
		\end{align*}
		Similarly, we have
		\begin{align*}
		(\hat{T}\tilde{\phi}_1(i)-\hat{T}\tilde{\phi}_2(i))\leq \max_{i\in \tilde{\mathscr{D}}_n}	\sup_{\nu\in \mathcal{P}(V(i))}\sup_{\mu\in \mathcal{P}(U(i))}e^{\bar{c}(i,\mu,\nu)}\|\tilde{\phi}_2-\tilde{\phi}_1\|_{\tilde{\mathscr{D}}_n}.
		\end{align*}
		Hence
		\begin{align*}
		\|\hat{T}\tilde{\phi}_1(i)-\hat{T}\tilde{\phi}_2(i)\|_{\tilde{\mathscr{D}}_n}\leq \max_{i\in \tilde{\mathscr{D}}_n}	\sup_{\nu\in \mathcal{P}(V(i))}\sup_{\mu\in \mathcal{P}(U(i))}e^{\bar{c}(i,\mu,\nu)}\|\tilde{\phi}_2-\tilde{\phi}_1\|_{\tilde{\mathscr{D}}_n},
		\end{align*}
		where for any function $f\in \mathcal{B}_{\tilde{\mathscr{D}}_n}$, $\|f\|_{\tilde{\mathscr{D}}_n}=\max\{|f(i)|:i\in \tilde{\mathscr{D}}_n\}$. Since $\bar{c}<0$, it is easy to see that $\displaystyle  \max_{i\in \tilde{\mathscr{D}}_n}\sup_{\nu\in \mathcal{P}(V(i))}\sup_{\mu\in \mathcal{P}(U(i))}e^{\bar{c}(i,\mu,\nu)}<1$.
		Hence $\hat{T}$ is a contraction map. Thus by Banach fixed point theorem, there exists a unique $\varphi\in \mathcal{B}_{\tilde{\mathscr{D}}_n}$ such that $\hat{T}(\varphi)=\varphi$.
		Now by applying Fan's minimax theorem in [\cite{Fan}, Theorem 3], we get
		\begin{align*}
		&\sup_{\nu\in P(V(i))}\inf_{\mu\in P(U(i))}\biggl[e^{\bar{c}(i,\mu,\nu)}\sum_{j\in S}\varphi(j)P(j|i,\mu,\nu)\biggr]=\inf_{\mu\in P(U(i))}\sup_{\nu\in P(V(i))}\biggl[e^{c(i,\mu,\nu)}\sum_{j\in S}\varphi(j)P(j|i,\mu,\nu)\biggr].
		\end{align*}
		Hence we conclude that (\ref{eq 2.7}) has unique solution.
		Now let $(\pi^{*1}_n,\pi^{*2}_n)\in \Pi^1_{SM}\times \Pi^2_{SM}$ be a mini-max selector of (\ref{eq 2.7}), i.e.,
		\begin{align}
		\varphi(i)&=\inf_{\mu\in \mathcal{P}(U(i))}\biggl[e^{\bar{c}(i,\mu,\pi^{*2}_n(i))}\sum_{j\in S}\varphi(j)P(j|i,\mu,\pi^{*2}_n(i))+g(i)\biggr]\nonumber\\
		&=\sup_{\nu\in \mathcal{P}(V(i))}\biggl[e^{\bar{c}(i,\pi^{*1}_n(i),\nu)}\sum_{j\in S}\varphi(j)P(j|i,\pi^{*1}_n(i),\nu)+g(i)\biggr].\label{eq 2.12}
		\end{align}
 		Now by Dynkin's formula \cite[Lemma~3.1]{WC5}, for any $(\pi^1,\pi^2)\in \Pi^1_{ad}\times \Pi^2_{ad}$ and $N\in\mathbb{N}$, we have
		\begin{align}
		&E^{\pi^1,\pi^2}_i\biggl[e^{\sum_{t=0}^{N\wedge{\tau}(\tilde{\mathscr{D}}_n)-1}\bar{c}(X_t,\pi^1_t,\pi^2_t)}\varphi(X_{N\wedge{\tau}(\tilde{\mathscr{D}}_n)})\biggr]-\varphi(i)\nonumber\\
		&=E^{\pi^1,\pi^2}_i\biggl[\sum_{t=1}^{N\wedge{\tau}(\tilde{\mathscr{D}}_n)}e^{\sum_{r=0}^{t-1}\bar{c}(X_r,\pi^1_r,\pi^2_r)}\bigg(\sum_{j\in S}\varphi(j)P(j|X_{t-1},\pi^1_{t-1},\pi^2_{t-1})-e^{\bar{c}(X_{t-1},\pi^1_{t-1},\pi^2_{t-1})}\varphi(X_{t-1})\bigg)\biggr].\label{eq 2.11}
		\end{align}
		Then, using (\ref{eq 2.12}) and (\ref{eq 2.11}), we obtain
		\begin{align*}
		 & E^{\pi^{*1}_n,\pi^2}_i\biggl[\sum_{t=0}^{N\wedge{\tau}(\tilde{\mathscr{D}}_n)-1}e^{\sum_{s=0}^{t-1}\bar{c}(X_s,\pi^{*1}_n(X_s),\pi^{2}_s)}g(X_t)\biggr] \\ & \leq -E^{\pi^{*1}_n,\pi^2}_i\biggl[e^{\sum_{s=0}^{N\wedge{\tau}(\tilde{\mathscr{D}}_n)-1}\bar{c}(X_s,\pi^{*1}_n(X_s),\pi^{2}_s)}\varphi(X_{N\wedge{\tau}(\tilde{\mathscr{D}}_n)})\biggr]+\varphi(i).
		\end{align*}
	Since $\bar{c}<0$ and $\varphi\in \mathcal{B}_{\tilde{\mathscr{D}}_n}$,  taking $N\rightarrow\infty$ in the above equation and using the dominated convergence theorem, we deduce that
		\begin{align*}
		& E^{\pi^{*1}_n,\pi^{2}}_i\biggl[\sum_{t=0}^{{\tau}(\tilde{\mathscr{D}}_n)-1}e^{\sum_{s=0}^{t-1}\bar{c}(X_s,\pi^{*1}_n(X_s),\pi^2_s)}g(X_t)\biggr]\\ &\leq -E^{\pi^{*1}_n,\pi^2}_i\biggl[e^{\sum_{s=0}^{{\tau}(\tilde{\mathscr{D}}_n)-1}\bar{c}(X_s,\pi^{*1}_n(X_s),\pi^{2}_s)}\varphi(X_{{\tau}(\tilde{\mathscr{D}}_n)})\biggr]+\varphi(i).
		\end{align*}
		Hence
		\begin{equation*}
		\varphi(i)\geq E^{\pi^{*1}_n,\pi^2}_i\biggl[\sum_{t=0}^{{\tau}(\tilde{\mathscr{D}}_n)-1}e^{\sum_{s=0}^{t-1}\bar{c}(X_s,\pi^{*1}_n(X_s),\pi^{2}_s)}g(X_t)\biggr].
		\end{equation*}
		Since $\pi^2\in \Pi^2$ is arbitrary,
		\begin{align}
		\varphi(i)&\geq\sup_{\pi^2\in \Pi^2_{ad}} E^{\pi^{*1}_n,\pi^2}_i\biggl[\sum_{t=0}^{{\tau}(\tilde{\mathscr{D}}_n)-1}e^{\sum_{s=0}^{t-1}\bar{c}(X_s,\pi^{*1}_n(X_s),\pi^{2}_s)}g(X_t)\biggr]\nonumber\\
		&\geq\inf_{\pi^1\in \Pi^1_{ad}}\sup_{\pi^2\in \Pi^2} E^{\pi^{1},\pi^2}_i\biggl[\sum_{t=0}^{{\tau}(\tilde{\mathscr{D}}_n)-1}e^{\sum_{s=0}^{t-1}\bar{c}(X_s,\pi^{1}_s,\pi^2_s)}g(X_t)\biggr].\label{eq 2.13}
		\end{align}
By similar arguments, using (\ref{eq 2.12}), (\ref{eq 2.11}) and the dominated convergence theorem, we obtain
		\begin{align}
		\varphi(i)&\leq\inf_{\pi^1\in \Pi^1_{ad}} E^{\pi^{1},\pi^{*2}_n}_i\biggl[\sum_{t=0}^{{\tau}(\tilde{\mathscr{D}}_n)-1}e^{\sum_{s=0}^{t-1}\bar{c}(X_s,\pi^{1}_s,\pi^{*2}_n(X_s))}g(X_t)\biggr]\nonumber\\
		&\leq \sup_{\pi^2\in \Pi^2_{ad}}\inf_{\pi^1\in \Pi^1_{ad}} E^{\pi^{1},\pi^2}_i\biggl[\sum_{t=0}^{{\tau}(\tilde{\mathscr{D}}_n)-1}e^{\sum_{s=0}^{t-1}\bar{c}(X_s,\pi^{1}_s,\pi^2_s)}g(X_t)\biggr].\label{eq 2.14}
		\end{align}
Now combining (\ref{eq 2.13}) and (\ref{eq 2.14}), we obtaiin (\ref{eq 2.8}).

\end{proof}

Next using Theorem~\ref{theo 2.1}, we show that for each $n\in \mathbb{N},$ Dirichlet eigenpair exists in $\tilde{\mathscr{D}}_n$\,. That is we establish the following  result.
\begin{lemma}\label{lemm 2.3}
Suppose Assumptions \ref{assm 2.2}, and \ref{assm 2.3} hold. Then there exists an eigenpair $(\rho_{n},\psi_{n})\in \mathbb{R}\times \mathcal{B}_{\tilde{\mathscr{D}}_n}^{+}$, $\psi_{n}\gneq 0$ for the following Dirichlet nonlinear eigenequation
	\begin{align}
e^{\rho_{n}}\psi_{n}(i)&=\inf_{\mu\in \mathcal{P}(U(i))}\sup_{\nu\in \mathcal{P}(V(i))}\bigg[e^{c(i,\mu,\nu)}\sum_{j\in S}\psi_{n}(j)P(j|i,\mu,\nu)\bigg]\nonumber\\
&=\sup_{\nu\in \mathcal{P}(V(i))}\inf_{\mu\in \mathcal{P}(U(i))}\bigg[e^{c(i,\mu,\nu)}\sum_{j\in S}\psi_{n}(j)P(j|i,\mu,\nu)\bigg].\label{eq 2.24}
	\end{align}
	 The eigenvalue of the above equation satisfies
	\begin{equation}
	\rho_{n}\leq\sup_{\pi^2\in \Pi^2_{ad}} \inf_{\pi^1\in \Pi^1_{ad}}\mathscr{J}^{\pi^1,\pi^2}(i,c),\label{eq 2.16}
	\end{equation}
	for all $i\in S$ such that $\psi_{n}(i)>0$. In addition we have that the sequence $\{\rho_{n}\}$ is bounded and $\displaystyle \liminf_{n\rightarrow \infty}\rho_n\geq 0$.
	\end{lemma}
\begin{proof}
Let $c^{'}=c-\|c\|_{\tilde{\mathscr{D}}_n}-\delta$ in $\tilde{\mathscr{D}}_n$,  for some constant $\delta>0$. Then it is easy to see that $c^{'}<-\delta$.
		Now  consider a mapping $\bar{T}:\mathcal{B}_{\tilde{\mathscr{D}}_n}\rightarrow\mathcal{B}_{\tilde{\mathscr{D}}_n}$ defined by
		\begin{align}
		\bar{T}(g)(i):=\sup_{\pi^2\in \Pi^2_{ad}}\inf_{\pi^1\in \Pi^1_{ad}} E^{\pi^1,\pi^2}_i\biggl[\sum_{t=0}^{{\tau}(\tilde{\mathscr{D}}_n)-1}e^{\sum_{s=0}^{t-1}c^{'}(X_s,\pi^1_s,\pi^2_s)}g(X_t)\biggr],~i\in \tilde{\mathscr{D}}_n\,,\label{eq 2.17}
		\end{align}
with	 $\bar{T}(g)(i)=0$ for $i\in \tilde{\mathscr{D}}_n^c,$ where $g\in \mathcal{B}_{\tilde{\mathscr{D}}_n}$. From Proposition \ref{prop 2.1} it is clear that $\bar{T}$ is well defined. Since ${c}^{'}<-\delta$, for $g_1,g_2\in \mathcal{B}_{\tilde{\mathscr{D}}_n}$, it follows that $$\|\bar{T}({g}_1)-	\bar{T}({g}_2)\|_{{\tilde{\mathscr{D}}_n}}\leq \alpha_1 \|{g}_1-{g}_2\|_{{\tilde{\mathscr{D}}_n}},$$ for some constant $\alpha_1>0$. Hence the map $\bar{T}$ is continuous.
		Let $g_1,~g_2\in \mathcal{B}_{\tilde{\mathscr{D}}_n}$ with $g_1\succeq g_2$. Also, let $\bar{T}(g_k)=\varphi_k$, $k=1,2$. Thus $\varphi_2$ is a solution of
		\begin{align*}
		\varphi_2(i)&=\sup_{\nu\in \mathcal{P}(V(i))}\inf_{\mu\in \mathcal{P}(U(i))}\biggl[e^{c^{'}(i,\mu,\nu)}\sum_{j\in S}\varphi_2(j)P(j|i,\mu,\nu)+g_2(i)\biggr]\nonumber\\
		&=\inf_{\mu\in \mathcal{P}(U(i))}\biggl[e^{c^{'}(i,\mu,\pi^{*2}_n(i))}\sum_{j\in S}\varphi_2(j)P(j|i,\mu,\pi^{*2}_n(i))+g_2(i)\biggr]~\forall i\in \tilde{\mathscr{D}}_n^c,
		\end{align*}
		where $\pi^{*2}_n\in \Pi^2_{SM}$ is an outer maximizing selector. Therefore
		\begin{align*}
		& \bar{T}(g_1)(i) - \bar{T}(g_2)(i)\\
		=&\sup_{\pi^2\in \Pi^2_{ad}}\inf_{\pi^1\in \Pi^1_{ad}}E^{\pi^1,\pi^2}_i\biggl[\sum_{t=0}^{{\tau}(\tilde{\mathscr{D}}_n)-1}e^{\sum_{s=0}^{t-1}c^{'}(X_s,\pi^1_s,\pi^2_s)}g_1(X_t)\biggr] \\ &-\sup_{\pi^2\in \Pi^2_{ad}}\inf_{\pi^1\in \Pi^1_{ad}}E^{\pi^1,\pi^2}_i\biggl[\sum_{t=0}^{{\tau}(\tilde{\mathscr{D}}_n)-1}e^{\sum_{s=0}^{t-1}c^{'}(X_s,\pi^1_s,\pi^2_s)}g_2(X_t)\biggr]\\
		=&\sup_{\pi^2\in \Pi^2_{ad}}\inf_{\pi^1\in \Pi^1_{ad}}E^{\pi^1,\pi^2}_i\biggl[\sum_{t=0}^{{\tau}(\tilde{\mathscr{D}}_n)-1}e^{\sum_{s=0}^{t-1}c^{'}(X_s,\pi^1_s,\pi^2_s)}g_1(X_t)\biggr] \\ &-\inf_{\pi^1\in \Pi^1_{ad}}E^{\pi^{1},\pi^{*2}_n}_i\biggl[\sum_{t=0}^{{\tau}(\tilde{\mathscr{D}}_n)-1}e^{\sum_{s=0}^{t-1}c^{'}(X_s,\pi^1_s,\pi^{*2}_n(X_s))}g_2(X_t)\biggr]\\
		\geq & \inf_{\pi^1\in \Pi^1_{ad}}E^{\pi^{1},\pi^{*2}_n}_i\biggl[\sum_{t=0}^{{\tau}(\tilde{\mathscr{D}}_n)-1}e^{\sum_{s=0}^{t-1}c^{'}(X_s,\pi^1_s,\pi^{*2}_n(X_s))ds}g_1(X_t)\biggr] \\ &-\inf_{\pi^1\in \Pi^1_{ad}}E^{\pi^{1},\pi^{*2}_n}_i\biggl[\sum_{t=0}^{{\tau}(\tilde{\mathscr{D}}_n)-1}e^{\sum_{s=0}^{t-1}c^{'}(X_s,\pi^1_s,\pi^{*2}_n(X_s))}g_2(X_t)\biggr]\\
		\geq & \inf_{\pi^1\in \Pi^1_{ad}}E^{\pi^{1},\pi^{*2}_n}_i\biggl[\sum_{t=0}^{{\tau}(\tilde{\mathscr{D}}_n)-1}e^{\sum_{s=0}^{t-1}c^{'}(X_s,\pi^1_s,\pi^{*2}_n(X_s))}(g_1(X_t)-g_2(X_t))\biggr].
		\end{align*}
		Hence $	\bar{T}(g_1)(i)-	\bar{T}(g_2)(i)\geq 0$ for all $i\in S$.
		This implies that $	\bar{T}(g_1)\succeq 	\bar{T}(g_2)$. Choose a function $g\in \mathcal{B}_{\tilde{\mathscr{D}}_n}$ such that $g(i_0)=1$ and $g(j)=0$ for all $j\neq i_0$, where $i_0$ is as in Assumption~\ref{assm 2.3}(ii). Thus by (\ref{eq 2.17}), we have
		$$	\bar{T}(g)(i_0)\geq g(i_0)>0.$$
Thus we have $\bar{T}(g)\succeq g$. Let $\{g_m\}\subset \mathcal{B}_{\tilde{\mathscr{D}}_n}$ be a bounded sequence. Then since $c^{'}<0$, from (\ref{eq 2.17}), we get $\|\bar{T}g_m\|_\infty\leq \alpha_2$, for some constant $\alpha_2>0$. So, by a diagonalization argument, there exists a subsequence $m_k$ of $m$ and a function $\phi\in \mathcal{B}_{\tilde{\mathscr{D}}_n}$ such that $\|\bar{T}g_{m_k}-\phi\|_{{\tilde{\mathscr{D}}_n}}\rightarrow 0$ as $k\rightarrow\infty$. Thus the map $\bar{T}$ is completely continuous. By the definition of the map $\bar{T}$, it is easy to see that $\bar{T}(\lambda g)=\lambda \bar{T}(g)$ for all $\lambda\geq 0$. Hence by Theorem \ref{theo 2.1}, there exists a nontrivial $\psi_{n}\in \mathcal{B}_{\tilde{\mathscr{D}}_n}^{+}$ and a constant $\lambda^{'}_{\tilde{\mathscr{D}}_n}>0$ such that $	\bar{T}(\psi_{n})=\lambda^{'}_{\tilde{\mathscr{D}}_n}\psi_{n}$, i.e.,
\begin{equation*}
\lambda^{'}_{\tilde{\mathscr{D}}_n}\psi_{n}(i) = \sup_{\nu\in \mathcal{P}(V(i))}\inf_{\mu\in \mathcal{P}(U(i))}\biggl[e^{c^{'}(i,\mu,\nu)}\sum_{j\in S}\lambda^{'}_{\tilde{\mathscr{D}}_n}\psi_{n}(j)P(j|i,\mu,\nu) + \psi_{n}(i)\biggr]~\forall i\in \tilde{\mathscr{D}}_n.
\end{equation*} Let $\rho_{n}^{'}=\log\biggl[\frac{\lambda^{'}_{\tilde{\mathscr{D}}_n}-1}{\lambda^{'}_{\tilde{\mathscr{D}}_n}}\biggr]$. Then
		\begin{align}
		e^{\rho^{'}_{n}}\psi_{n}(i)=\sup_{\nu\in \mathcal{P}(V(i))}\inf_{\mu\in \mathcal{P}(U(i))}\biggl[e^{c^{'}(i,\mu,\nu)}\sum_{j\in S}\psi_{n}(j)P(j|i,\mu,\nu)\biggr]~\forall i\in \tilde{\mathscr{D}}_n.\label{EE}
		\end{align}
Now multiplying both sides of (\ref{EE}) by $e^{\|c\|_{\tilde{\mathscr{D}}_n} \,+\, \delta}$ and applying Fan's minimax theorem, (see [\cite{Fan}, Theorem 3]), we obtain
		\begin{align*}
		e^{	\rho_{n}}\psi_{n}(i)&=\sup_{\nu\in \mathcal{P}(V(i))}\inf_{\mu\in \mathcal{P}(U(i))}\biggl[e^{c(i,\mu,\nu)}\psi_{n}(j)P(j|i,\mu,\nu)\biggr]\\
		&=\inf_{\mu\in \mathcal{P}(U(i))}\sup_{\nu\in \mathcal{P}(V(i))}\biggl[e^{c(i,\mu,\nu)}\psi_{n}(j)P(j|i,\mu,\nu)\biggr]~\forall i\in \tilde{\mathscr{D}}_n\,,
		\end{align*}
where $\rho_{n}=\rho^{'}_{n}+{\|c\|_{\tilde{\mathscr{D}}_n}+\delta}$\,.

			Since $\psi_n\geq 0$ and $\psi_n(i_0)>0$, it follows that from the above equation that $\rho_{n}\geq 0$. Now if $\rho_{n}=0$, it is easy to show that (\ref{eq 2.16}). Assume that $\rho_{n}>0$. Let $\pi^{*2}_n\in \Pi_{SM}^2$ be an outer maximizing selector of (\ref{eq 2.24}). Then we have
		\begin{align}
		e^{\rho_{n}}\psi_{n}(i)=\inf_{\mu\in \mathcal{P}(U(i))}\biggl[e^{c(i,\mu,\pi^{*2}_n(i))}\sum_{j\in S}\psi_{n}(j)P(j|i,\mu,\pi^{*2}_n(i))\biggr]~\forall i\in\tilde{\mathscr{D}}_n.\label{eq 2.19}
		\end{align}
Therefore by using Dynkin's formula and  (\ref{eq 2.19}), we obtain
		\begin{align}
		\psi_{n}(i)&\leq E^{\pi^1,\pi^{*2}_n}_i\biggl[e^{\sum_{s=0}^{T-1}(c(X_s,\pi^1_s,\pi^{*2}_n(X_s))-\rho_{n})}\psi_{n}(X_{T})I_{\{T<{\tau}(\tilde{\mathscr{D}}_n)\}}\biggr]\nonumber\\
		&\leq (\sup_{\tilde{\mathscr{D}}_n}\psi_{n})E^{\pi^1,\pi^{*2}_n}_i\biggl[e^{\sum_{s=0}^{T-1}(c(X_s,\pi^1_s,\pi^{*2}_n(X_s))-\rho_{n})}\biggr].\label{eq 2.20}
		\end{align}
Now, taking logarithm on the both sides of (\ref{eq 2.20}), dividing by $T$ and letting $T\rightarrow\infty$, for each $i\in S$ for which $\psi_{n}>0$, we deduce that
		\begin{align*}
		\rho_{n}&\leq \mathscr{J}^{\pi^1,\pi^{*2}_n}(i,c).
		\end{align*}
		Since $\pi^1\in \Pi^1_{ad}$ is arbitrary, we get
		\begin{align*}
		\rho_{n}&\leq \inf_{\pi^1\in \Pi^1_{ad}}\mathscr{J}^{\pi^1,\pi^{*2}_n}(i,c)\leq \sup_{\pi^2\in \Pi^2_{ad}}\inf_{\pi^1\in \Pi^1_{ad}}\mathscr{J}^{\pi^{1},\pi^{2}}(i,c).
		\end{align*}
		Now we prove that $\rho_{n}$ is bounded. Under Assumption 2.2 (a) since $\|c\|_{\infty} < \tilde{\gamma}$, it is easy to see that $\mathscr{J}^{\pi^1,\pi^2}(i,c) \leq \tilde{\gamma}$\,. Under Assumption 2.2 (b) since $\tilde{\mathscr{K}}$ is finite, there exists a constant $k_1$ such that (\ref{eq 2.6}) can be written as
		\begin{align}
		\sup_{(u,v)\in U(i)\times V(i)}\sum_{j\in S}\mathcal{W}(j)P(j|i,u,v)\leq e^{(k_1-\tilde{\ell}(i))}\mathcal{W}(i) ~\forall i\in S.\label{eq 2.25}
		\end{align}
		Then by using (\ref{eq 2.2}) and successive conditioning, we get
		\begin{align}
		E^{\pi^1,\pi^2}_i\bigg[e^{\sum_{t=0}^{T-1}(\tilde{\ell}(X_t)-k_1)}\mathcal{W}(X_T)\bigg]\leq \mathcal{W}(i)~\forall i\in S.\label{eq 2.26}
		\end{align}
		Since, $\mathcal{W}\geq 1$, from (\ref{eq 2.26}), we get
		$$\mathscr{J}^{\pi^1,\pi^2}(i,\tilde{\ell})\leq k_1~\text{ for all}~i\in S.$$
		Now since $\displaystyle \tilde{\ell}-\sup_{(u,v)\in U(i)\times V(i)}c(\cdot,u,v)$ is norm-like, there exists a constant $k_2$ such that for all $i\in S$, we have $\displaystyle \sup_{(u,v)\in U(i)\times V(i)}c(i,u,v)\leq\tilde{\ell}(i)+k_2$. Hence we get
		\begin{align}
		\mathscr{J}^{\pi^1,\pi^2}(i,c)\leq k_1+k_2~~\forall (\pi^1,\pi^2)\in \Pi^1_{ad}\times \Pi^2_{ad}, \forall i\in S.\label{eq 2.27}
		\end{align}
Therefore it is clear that $\rho_n$ has an upper bound.

Next we show that $\rho_n$ is bounded below. If not, then along a subsequence $\rho_n\rightarrow-\infty$ as $n\rightarrow\infty$. So, $\rho_n<0$ for all large enough $n$. By Assumption \ref{assm 2.3}(ii) from (\ref{eq 2.24}), it is easy to see that $\psi_{n}(i_0)>0$ and so normalizing $\psi_{n}$, we have $\psi_{n}(i_0)=1$. Let $(\pi_n^{*1},\pi_n^{*2})\in \Pi_{SM}^1\times \Pi_{SM}^2$ be a mini-max selector of (\ref{eq 2.24}), thus we have
		\begin{align}
		1=\psi_{n}(i_0)&=e^{-\rho_n}\sup_{\nu\in \mathcal{P}(V(i_0))}\biggl[e^{c(i_0,\pi_n^{*1}(i_0),\nu)}\sum_{j\in S}\psi_n(j)P(j|i_0,\pi_n^{*1}(i_0),\nu)\biggr]\nonumber\\
		&=e^{-\rho_n}\biggl[e^{c(i_0,\pi_n^{*1}(i_0),\pi_n^{*2}(i_0))}\sum_{j\in S}\psi_n(j)P(j|i_0,\pi_n^{*1}(i_0),\pi_n^{*2}(i_0))\biggr].\label{L2.2E2.26}
		\end{align}
		Since $c(i_0,\pi_n^{*1}(i_0),\nu)-\rho_n>0$ for all large $n$, we get
		\begin{align}
		\sum_{j\in S}\psi_n(j)P(j|i_0,\pi_n^{*1}(i_0),\nu)\leq 1.\label{L2.2E2.27}
		\end{align}
Thus, in view of  Assumption \ref{assm 2.3}(ii), we obtain
		\begin{align}
		\psi_n(j)\leq \sup_{(\mu,\nu)\in \mathcal{P}(U(i_0))\times \mathcal{P}(V(i_0))}\frac{1}{P(j|i_0,\mu,\nu)}~\text{ for }~j\neq i_0.\label{L2.2E2.2.9}
		\end{align}
This implies that, $\psi_n$ has an upper bound. Thus by a standard diagonalization argument, there exists a subsequence (by an abuse of notation denoting by the same sequence) and a bounded function $\psi\geq 0$ with $\psi(i_0)=1$ such that $\psi_{n}(i)\rightarrow \psi(i)$, as $n\rightarrow \infty$ for all $i\in S$.
		Now, since $\Pi_{SM}^1$ and $\Pi_{SM}^2$ are compact, there exist a further subsequence and $\pi^{*1}\in \Pi_{SM}^1$ and $\pi^{*2}\in \Pi_{SM}^2$, such that $\pi_n^{*1} \rightarrow \pi^{*1}$ and $\pi_n^{*2} \rightarrow \pi^{*2}$ as $n\rightarrow\infty$.
		 Since $c\geq 0$, by (\ref{eq 2.24}), we get
		\begin{equation}
		e^{\rho_n}\psi_n(i)\geq \biggl[\sum_{j\in S}\psi_n(j)P(j|i,\pi^{*1}_n(i),\pi^{*2}_n(i))\biggr].
		\end{equation}
Hence, by taking $n\rightarrow\infty$, it follows that
		\begin{equation}
		\sum_{j\in S}\psi(j)P(j|i,\pi^{*1}(i),\pi^{*2}(i))\leq 0.\label{L2.2E2.32L}
		\end{equation}
		Therefore by using Dynkin formula and the fact that $\psi\geq 0$, we have
	$$	\psi(i)\geq E^{\pi^{*1},\pi^{*2}}_i[\psi(X_t)]~\forall i\in S.$$
		Hence, $\{\psi(X_n),\mathscr{F}_n\}$ is supermartingale where $X_t$ is the Markov process under the pair of stationary strategies $(\pi^{*1},\pi^{*2})\in \Pi^1_{SM}\times \Pi^2_{SM}$. So, by Doob's martingale convergence theorem $\psi(X_n)\rightarrow \hat{Y}$ almost surely, as $n\rightarrow\infty$. On the other hand by Assumption \ref{assm 2.2}, we have $X_t$ is recurrent. Hence $X_t$ visits every state (in particular $i_0$)  of $S$ infinitely often. So, $\psi(X_n)$ converges only if $\psi\equiv 1$. But putting $i=i_0$ in (\ref{L2.2E2.32L}) and using Assumption \ref{assm 2.3}, we get $\psi(i_0)=1$ and $\psi(j)=0$ for $j\neq i_0$. So, we arrive at a contradiction. This implies that, $\rho_n$ is bounded below.
		
		Now we show that $\displaystyle  \hat{\rho}=\liminf_{n\rightarrow \infty}\rho_{n}\geq 0$. If not, then on contrary, $\hat{\rho}<0$. So, along some subsequence, we have (with an abuse of notation, we use the same sequence) $\displaystyle \rho_{n}\rightarrow\hat{\rho}$, as $n\rightarrow\infty$ and for large $n$, $\rho_n<0$. So, for large $n$, $c(i,\mu,\nu)-\rho_n>0$ for all $(\mu,\nu)\in \mathcal{P}(U(i))\times \mathcal{P}(V(i))$. So, by repeating the above arguments, there exist a subsequence (by an abuse of notation, we take the same sequence) and a bounded function $\phi:S\rightarrow\mathbb{R}$ such that $\psi_n\rightarrow\phi$ satisfying $\phi(i_0)=1$.
		Now from (\ref{eq 2.24}), we get
		\begin{equation}
		\psi_{n}(i)
		\geq \biggl[\sum_{j\in S}\psi_n(j){P}(j|i,\pi^{*1}_n(i),\pi_n^{*2}(i))\biggr].\label{L2.2E2.33}
		\end{equation}
		By Fatou's lemma, taking $n\rightarrow\infty$, we deduce that
	$$	\phi(i)\geq E^{\pi^{*1},\pi^{*2}}_i[\phi(X_1)]~\forall i\in S,$$
		for some pair of stationary strategies $(\pi^{*1},\pi^{*2})\in \Pi^1_{SM}\times \Pi^2_{SM}$. So, $\{\phi(X_t)\}$ is supermartingale. Thus by similar arguments as above we get $\phi\equiv 1$. Now, taking limit $n\rightarrow\infty$ in (\ref{eq 2.24}), we obtain
		$$1=\phi(i)\geq e^{c(i,\pi^{*1}(i),\pi^{*2}(i))-\hat{\rho}}>1.$$
		But this is a contradiction. Thus, $\displaystyle \liminf_{n\rightarrow \infty}\rho_n\geq 0$.

	\end{proof}
\section{Existence of risk-sensitive average optimal strategies}
In this section we prove the existence of a risk-sensitive average optimal stationary strategy using the Shapley equation.  Now we state and prove our main result of this section.
\begin{thm}\label{theo 2.2}
Suppose Assumptions \ref{assm 2.2} and \ref{assm 2.3} hold. Then there exists an eigenpair $(\rho^*,\psi^{*})\in \mathbb{R}_+\times L^\infty_\mathcal{W}$ with $\psi^{*}>0$, such that
\begin{align}
e^{\rho^*}\psi^{*}(i)&=\sup_{\nu\in \mathcal{P}(V(i))}\inf_{\mu\in \mathcal{P}(U(i))}\bigg[e^{c(i,\mu,\nu)}\sum_{j\in S}\psi^{*}(j)P(j|i,\mu,\nu)\bigg]\nonumber\\
&=\inf_{\mu\in \mathcal{P}(U(i))}\sup_{\nu\in \mathcal{P}(V(i))}\bigg[e^{c(i,\mu,\nu)}\sum_{j\in S}\psi^{*}(j)P(j|i,\mu,\nu)\bigg], ~i\in S.\label{T2.21}
\end{align}
Moreover, we have the following
\begin{enumerate}
	\item [(i)]
	\begin{equation}
		\rho^{*}=\inf_{i\in S}\sup_{\pi^2\in \Pi^2_{ad}}\inf_{\pi^1\in \Pi^1_{ad}}\mathscr{J}^{\pi^1,\pi^2}(i,c)=\inf_{i\in S}\inf_{\pi^1\in \Pi^1_{ad}}\sup_{\pi^2\in \Pi^2_{ad}}\mathscr{J}^{\pi^1,\pi^2}(i,c).\label{TE2.13}
	\end{equation}
	
	\item [(ii)] If $(\pi^{*1},\pi^{*2})\in \Pi^1_{SM} \times \Pi^2_{SM} $ be a mini-max selector of (\ref{T2.21}),
	then $(\pi^{*1},\pi^{*2})\in \Pi^1_{SM} \times \Pi^2_{SM} $  is a saddle point equilibrium, i.e.,
		\begin{align}
	\mathscr{J}^{{\pi}^{*1},\pi^{2}}(i,c) \leq  \mathscr{J}^{{\pi}^{*1},{\pi}^{*2}}(i,c) = \rho^{*} \leq \mathscr{J}^{\pi^1,{\pi}^{*2}}(i,c),\,~\forall \pi^1\in \Pi^1_{ad},~\forall \pi^2\in \Pi^2_{ad}.\label{T2.23}
	\end{align}
Thus the value of the game is independent of the initial state.
	\item [(iv)] Let $(\pi^{*1},\pi^{*2})\in \Pi_{SM}^1\times \Pi_{SM}^2$ is a saddle point equilibrium, then  this pair is a mini-max selector of (\ref{T2.21}).
\end{enumerate}
\end{thm}
Rest of this section is dedicated to the proof of Theorem~\ref{theo 2.2}.

Since $c\geq 0$, using Assumption \ref{assm 2.2}, there exists a finite set $\hat{\mathscr{B}}$ containing $\tilde{\mathscr{K}}$ such that we have the  following:
	\begin{itemize}
		\item  Under Assumption  \ref{assm 2.2} (a): since $\tilde{\gamma}> \|{c}\|_{\infty}$, from (\ref{eq 2.16}) we have $\rho_{n} \leq \tilde{\gamma}$. Thus, for all large enough $n$ it holds that
		\begin{align}
		(\sup_{(u,v)\in U(i)\times V(i)}{c}(i,u,v)-\rho_{n})<\tilde{\gamma}~\forall i\in {\hat{\mathscr{B}}}^c.\label{eq 2.32}
		\end{align}
		\item  Under Assumption  \ref{assm 2.2} (b): since the function $\displaystyle \ell(\cdot)-\max_{(u,v)\in U(\cdot)\times V(\cdot)}{c}(\cdot,u,v)$ is norm-like, for all large enough $n$ it holds that
		\begin{align}
		(\sup_{(u,v)\in U(i)\times V(i)}{c}(i,u,v)-\rho_n)<\tilde{\ell} (i)~\forall i\in {\hat{\mathscr{B}}}^c.\label{eq 2.33}
		\end{align}
	\end{itemize}
	Now letting $n\to \infty$ from (\ref{eq 2.24}) we show that the limiting equation admits a positive eigenpair.
\begin{lemma}\label{L2.4}
Suppose Assumptions \ref{assm 2.2} and \ref{assm 2.3} hold. Then there exists an eigenpair $(\rho^*,\psi^{*})\in \mathbb{R}_+\times L^\infty_{\mathcal{W}}$ with $\psi^{*}>0$, such that
	\begin{align}
	e^{\rho^*}\psi^{*}(i)&=\sup_{\nu\in \mathcal{P}(V(i))}\inf_{\mu\in \mathcal{P}(U(i))}\bigg[e^{c(i,\mu,\nu)}\sum_{j\in S}\psi^{*}(j)P(j|i,\mu,\nu)\bigg]\nonumber\\
	&=\inf_{\mu\in \mathcal{P}(U(i))}\sup_{\nu\in \mathcal{P}(V(i))}\bigg[e^{c(i,\mu,\nu)}\sum_{j\in S}\psi^{*}(j)P(j|i,\mu,\nu)\bigg], ~i\in S.\label{eq 2.30}
	\end{align}
	Furthermore, for any mini-max selector $(\pi^{*1},\pi^{*2})\in \Pi^1_{SM} \times \Pi^2_{SM} $ of (\ref{eq 2.30}) we have the following:
	\begin{itemize}
    \item[(i)]		
    \begin{equation}
	\rho^*\leq \inf_{i\in S}\sup_{\pi^2\in \Pi^2_{ad}}\inf_{\pi^1\in \Pi^1_{ad}}\mathscr{J}^{\pi^1,\pi^2}(i,c).\label{L2.3E2.31}
		\end{equation}
\item[(ii)] For any finite set $\hat{\mathscr{B}}_1\supset \hat{\mathscr{B}}$, we have the following stochastic representation of the eigenfunction
		\begin{align}
		\psi^{*}(i)
		&=\inf_{\pi^1\in\Pi^1_{ad}}E^{\pi^{1},\pi^{*2}}_i\bigg[e^{\sum_{t=0}^{\check{\tau}(\hat{\mathscr{B}}_1)-1}(c(X_t,\pi^{1}_t,\pi^{*2}(X_t))-\rho^*)}\psi^{*}(X_{\check{\tau}(\hat{\mathscr{B}}_1)})\bigg]\nonumber\\
		&=\sup_{\pi^2\in \Pi^2_{ad}}E^{\pi^{*1},\pi^{2}}_i\bigg[e^{\sum_{t=0}^{\check{\tau}(\hat{\mathscr{B}}_1)-1}(c(X_t,\pi^{*1}(X_t),\pi^{2}_t)-\rho^*)}\psi^{*}(X_{\check{\tau}(\hat{\mathscr{B}}_1)})\bigg]~\forall i\in {\hat{\mathscr{B}}_1}^c.\label{eq 2.31}
		\end{align}
	\end{itemize}
\end{lemma}
\begin{proof}
	First we scale $\psi_n$ in such a way that we obtain $\psi_n(i)\leq \mathcal{W}(i)$ for all $i\in S$. Set
$$\tilde{\theta}_n=\sup\{k>0:(\mathcal{W}-k\psi_n)>0~\text{in}~S\}.$$
Since $\psi_{n}$ vanishes in ${\tilde{\mathscr{D}}_n}^c$ and $\psi_{n}\gneq 0$, it follows that $\hat{\theta}_n$ is finite. We claim that if we replace $\psi_{n}$ by $\tilde{\theta}_n\psi_{n}$, then $\psi_{n}$ touches $\mathcal{W}$ inside $\hat{\mathscr{B}}$. If this is not true, then on the contrary, we assume that for some state $\hat{i}\in {\hat{\mathscr{B}}}^c \cap \tilde{\mathscr{D}}_n$, $(\mathcal{W}-\psi_{n})(\hat{i})=0$ and $\mathcal{W}-\psi_{n}>0$ in $\hat{\mathscr{B}}\cup \tilde{\mathscr{D}}_n^c$.
Let $\pi^{*2}_n$ be an outer maximizing selector of (\ref{eq 2.24}).
Then under Assumption \ref{assm 2.2} (b), applying Dynkin's formula (as in \cite[Lemma~3.1]{WC5}), we obtain
\begin{align*}
\psi_n(\hat{i})&\leq E^{\pi^1,\pi^{*2}_n}_{\hat{i}}\biggl[e^{\sum_{s=0}^{N\wedge\check{\tau}(\hat{\mathscr{B}})-1}(c(X_s,\pi^1_s,\pi^{*2}_n(X_s))-\rho_n)}\psi_n(X_{N\wedge \check{\tau}(\hat{\mathscr{B}})})I_{\{N\wedge\check{\tau}(\hat{\mathscr{B}})<{\tau}(\tilde{\mathscr{D}}_n)\}}\biggr]\\
&\leq E^{\pi^1,\pi^{*2}_n}_{\hat{i}}\biggl[e^{\sum_{s=0}^{N\wedge\check{\tau}(\hat{\mathscr{B}})-1}\tilde{\ell}(X_s)}\psi_n(X_{N\wedge \check{\tau}(\hat{\mathscr{B}})})I_{\{N\wedge\check{\tau}({\hat{\mathscr{B}}})<{\tau}(\tilde{\mathscr{D}}_n)\}}\biggr].
\end{align*}
Since $\psi_n\leq \mathcal{W}$ (by our scaling), in view of Lemma \ref{lemm 2.4}, by the dominated convergence theorem taking $N\rightarrow\infty$, we get
\begin{align*}
\psi_n(\hat{i})
\leq E^{\pi^1,\pi^{*2}_n}_{\hat{i}}\biggl[e^{\sum_{s=0}^{\check{\tau}(\hat{\mathscr{B}})-1}\tilde{\ell}(X_s)ds}\psi_n(X_{ \check{\tau}(\hat{\mathscr{B}})})\biggr].
\end{align*}
Combining this and (\ref{eq 2.22}), we get
\begin{align*}
0=(\mathcal{W}-\psi_n)(\hat{i})\geq E^{\pi^1,\pi^{*2}_n}_{\hat{i}}\bigg[e^{\sum_{s=0}^{\check{\tau}(\hat{\mathscr{B}})-1}\tilde{\ell}(X_s) ds}(\mathcal{W}-\psi_n)(X_{\check{\tau}(\hat{\mathscr{B}})})\bigg]>0.
\end{align*}
But this is a contradiction. Thus $\psi_n$ touches $\mathcal{W}$ inside $\hat{\mathscr{B}}$.
So, there exists a point  $i^*\in\hat{\mathscr{B}}$ such that $(\mathcal{W}-\psi_{n})(i^*)=0$, for all large $n$. Since $\psi_{n}\leq \mathcal{W}$, by diagonalization arguments, there exist a subsequence (here we use the same sequence by an abuse of notation), and a function $\psi^{*}\leq \mathcal{W}$ such that $\psi_{n}\rightarrow\psi^{*}$ as $n\rightarrow \infty$. Again, from Lemma \ref{lemm 2.3}, we know that the sequence $\{\rho_{n}\}$ is bounded and $\displaystyle \liminf_{n\rightarrow \infty}\rho_{n}\geq 0$, thus along a further subsequence we have $\rho_{n}\rightarrow \rho^*$ as $n\rightarrow \infty$ for some $\rho^* \geq 0$. Also, we have $(\mathcal{W}-\psi^{*})(\hat{i}^*)=0$ for some $\hat{i}^*\in \hat{\mathscr{B}}$. By the continuity-compactness assumptions, there exists a mini-max selector $(\pi_n^{*1},\pi_n^{*2})\in\Pi_{SM}^1\times \Pi_{SM}^2$  such that from (\ref{eq 2.24}), we get
\begin{align}
e^{\rho_{n}}\psi_{n}(i)&=\sup_{\nu\in \mathcal{P}(V(i))}\bigg[e^{c(i,\pi^{*1}_n(i)),\nu)}\sum_{j\in S}\psi_{n}(j)P(j|i,\pi^{*1}_n(i),\nu)\bigg]\nonumber\\
&=\inf_{\mu\in \mathcal{P}(U(i))}\bigg[e^{c(i,\mu,\pi^{*2}_n(i))}\sum_{j\in S}\psi_{n}(j)P(j|i,\mu,\pi^{*2}_n(i))\bigg].\label{T3.1E3.9}
\end{align}
Note that since $\psi_n\in L^\infty_{\mathcal{W}}$, we have
\begin{equation}\label{ForDomin1}
\sum_{j\in S}\psi_{n}(j){P}(j|i,u,v)\leq \sum_{j\in S}\mathcal{W}(j){P}(j|i,u,v)~\forall (i,u,v)\in \mathcal{K}.
\end{equation}
Since $\Pi_{SM}^1$ and $\Pi_{SM}^2$ are compact there exists $(\pi^{*1},\pi^{*2})\in\Pi_{SM}^1\times \Pi_{SM}^2$ such that $\pi^{*1}_n\rightarrow\pi^{*1}$ and $\pi^{*2}_n\rightarrow\pi^{*2}$ as $n\rightarrow\infty$.
Now from (\ref{T3.1E3.9}) we obtain,
\begin{equation}
e^{\rho_{n}}\psi_{n}(i)\geq \bigg[e^{c(i,\pi^{*1}_n(i),\nu)}\sum_{j\in S}\psi_{n}(j)P(j|i,\pi^{*1}_n(i),\nu)\bigg].\label{L2.3E2.36}
\end{equation}
Then, taking $n\rightarrow\infty$ from (\ref{L2.3E2.36})and ,using Lemma \ref{lemm 3.1}, the extended Fatou's lemma [\cite{HL}, Lemma 8.3.7], , we obtain
$$e^{\rho^*}\psi^{*}(i)\geq e^{c(i,\pi^{*1}(i),\nu)}\sum_{j\in S}\psi^{*}(j)P(j|i,\pi^{*1}(i),\nu).$$
Thus
\begin{align}
e^{\rho^*}\psi^{*}(i)&\geq \sup_{\nu\in \mathcal{P}(V(i))} \bigg[e^{c(i,\pi^{*1}(i),\nu)}\sum_{j\in S}\psi^{*}(j)P(j|i,\pi^{*1}(i),\nu)\biggr]\nonumber\\
&\geq\inf_{\mu\in \mathcal{P}(U(i))} \sup_{\mu\in \mathcal{P}(V(i))}\bigg[e^{c(i,\mu,\nu)}\sum_{j\in S}\psi^{*}(j)P(j|i,\mu,\nu)\biggr].\label{eq 2.34}
\end{align}
Also, from (\ref{T3.1E3.9}), we get
\begin{align*}
e^{\rho_{n}}\psi_{n}(i)\leq\bigg[e^{c(i,\mu,\pi^{*2}_n(i))}\sum_{j\in S}\psi_{n}(j)P(j|i,\mu,\pi^{*2}_n(i))\bigg].
\end{align*}
Using (\ref{ForDomin1}), by the dominated convergence theorem taking limit $n\rightarrow\infty$ in above equation, we deduce that
\begin{align}
e^{\rho^*}\psi^{*}(i)&\leq  \inf_{\mu\in \mathcal{P}(U(i))}\bigg[e^{c(i,\mu,\pi^{*2}(i))}\sum_{j\in S}\psi^{*}(j)P(j|i,\mu,\pi^{*2}(i))\biggr]\nonumber\\
& \leq \sup_{\nu\in \mathcal{P}(V(i))}\inf_{\mu\in \mathcal{P}(U(i))}\bigg[e^{c(i,\mu,\nu)}\sum_{j\in S}\psi^{*}(j)P(j|i,\mu,\nu)\biggr].\label{eq 2.35}
\end{align}
Hence by (\ref{eq 2.34}) and (\ref{eq 2.35}), we get (\ref{eq 2.30}).
Since we have $(\mathcal{W}-\psi^*)(\hat{i}^*)=0$ and $\mathcal{W}\geq 1$, it follows that $\psi^{*}$ is nontrivial. Next, we claim that $\psi^{*}>0$. If not, then on contrary there exists a point $\tilde{i}\in S$ for which $\psi^{*}(\tilde{i})=0$. Again by continuity-compactness assumptions, there exists a mini-max selector $(\pi^{*1},\pi^{*2})$ such that (\ref{eq 2.30}) can be rewritten as
\begin{equation*}
e^{\rho^*}\psi^{*}({i})=\bigg[e^{c({i},\pi^{*1}(i),\pi^{*2}(i))}\sum_{j\in S}\psi^{*}(j)P(j|{i},\pi^{*1}(i),\pi^{*2}(i))\biggr]~~~\forall i\in S.
\end{equation*}
So, we get
\begin{equation}\label{eq 2.37}
0 = e^{\rho^*}\psi^{*}(\tilde{i})=\bigg[e^{c(\tilde{i},\pi^{*1}(\tilde{i}),\pi^{*2}(\tilde{i}))}\sum_{j\in S}\psi^{*}(j)P(j|\tilde{i},\pi^{*1}(\tilde{i}),\pi^{*2}(\tilde{i}))\biggr].
\end{equation}
Since $\psi^{*}$ in nontrivial, there exists $\hat{i}\in S$ such that $\psi^{*}(\hat{i})>0$.
Again, since $X$ is irreducible under any pair of strategies $(\pi^{*1},\pi^{*2})\in \Pi^1_{SM}\times \Pi^2_{SM}$, there exists $i_1,i_2,\cdots,i_n\in S$ satisfying
\begin{align*}
P(\hat{i}|i_n,\pi^{*1}(i_n),\pi^{*2}(i_n))P(i_n|i_{n-1},\pi^{*1}(i_{n-1}),\pi^{*2}(i_{n-1}))\cdots P(i_1|\tilde{i},\pi^{*1}(\tilde{i}),\pi^{*2}(\tilde{i}))>0.
\end{align*}
Thus, from (\ref{eq 2.37}) we deduce that $\psi^*(\hat{i})=\psi^*(i_1)=\cdots=\psi^*(i_n)=\psi^*(\tilde{i})=0$.
But this contradicts to the fact that $\psi^{*}$ is nontrivial. This establishes our claim.

Next we prove (\ref{L2.3E2.31}). Since $\psi^{*}>0$ and $\psi_n(i)\rightarrow \psi^{*}(i)$ as $n\rightarrow \infty$. Hence for all  $n$ large enough, $\psi_{n}>0$. Therefore using (\ref{eq 2.16}), we have $\displaystyle \rho^{*}= \lim_{n\rightarrow \infty}\rho_{n}\leq \sup_{\pi^2\in \Pi^2_{ad}}\inf_{\pi^1\in \Pi^1_{ad}}\mathscr{J}^{\pi^1,\pi^2}(i,c)$ for all $i\in S$.

Finally we prove the stochastic representation (\ref{eq 2.31}) of $\psi^*$. As before there exists a pair of strategies $(\pi^{*1},\pi^{*2})\in \Pi_{SM}^1\times \Pi_{SM}^2$ satisfying
	\begin{align}
	e^{\rho^*}\psi^{*}(i)&= \sup_{\nu\in \mathcal{P}(V(i))}\bigg[e^{c(i,\pi^{*1}(i),\nu)}\sum_{j\in S}\psi^{*}(j)P(j|i,\pi^{*1}(i),\nu)\biggr]\nonumber\\
	&=\inf_{\mu\in \mathcal{P}(U(i))} \bigg[e^{c(i,\mu,\pi^{*2}(i))}\sum_{j\in S}\psi^{*}(j)P(j|i,\mu,\pi^{*2}(i))\biggr].\label{eq 2.38}
	\end{align}	
	Now for any finite set $\hat{\mathscr{B}}_1 \supset \hat{\mathscr{B}}$, applying Dynkin's formula (as in \cite[Lemma~3.1]{WC5}) from (\ref{eq 2.38}), we get
	\begin{equation*}
	\psi^{*}(i)\leq E^{\pi^{1},\pi^{*2}}_i\bigg[e^{\sum_{t=0}^{\check{\tau}(\hat{\mathscr{B}}_1)\wedge N-1}(c(X_t,\pi^{1}_t,\pi^{*2}(X_t))-\rho^*)}\psi^{*}(X_{\check{\tau}(\hat{\mathscr{B}}_1)\wedge N})\bigg]~\forall i\in{\hat{\mathscr{B}}_1}^c.
	\end{equation*}
	Since $\psi^{*}\leq \mathcal{W}$, using estimates of Lemma \ref{lemm 2.4}, by the dominated convergence theorem taking $N\rightarrow\infty$, it follows that
	\begin{equation}
	\psi^{*}(i)\leq  E^{\pi^{1},\pi^{*2}}_i\bigg[e^{\sum_{t=0}^{\check{\tau}(\hat{\mathscr{B}}_1)-1}(c(X_t,\pi^{1}_t,\pi^{*2}(X_t))-\rho^*)}\psi^{*}(X_{\check{\tau}(\hat{\mathscr{B}}_1)})\bigg]~\forall i\in {\hat{\mathscr{B}}_1}^c.\label{eq 2.41}
	\end{equation}
	Hence
	\begin{align}
	\psi^{*}(i)&\leq \inf_{\pi^1\in \Pi^1_{ad}} E^{\pi^{1},\pi^{*2}}_i\bigg[e^{\sum_{t=0}^{\check{\tau}(\hat{\mathscr{B}}_1)-1}(c(X_t,\pi^{1}_t,\pi^{*2}(X_t))-\rho^*)}\psi^{*}(X_{\check{\tau}(\hat{\mathscr{B}}_1)})\bigg]\nonumber\\
	&\leq  \sup_{\pi^2\in \Pi^2_{ad}}\inf_{\pi^1\in\Pi^1_{ad}}E^{\pi^1,\pi^{2}}_i\bigg[e^{\sum_{t=0}^{\check{\tau}(\hat{\mathscr{B}}_1)-1}(c(X_t,\pi^1_t,\pi^{2}_t)-\rho^*)}\psi^{*}(X_{\check{\tau}(\hat{\mathscr{B}}_1)})\bigg],~\forall i\in {\hat{\mathscr{B}}_1}^c.\label{eq 2.42}
	\end{align}
	Now using (\ref{eq 2.38}) and Dynkin's formula
	\begin{equation*}
	\psi^{*}(i)\geq E^{\pi^{*1},\pi^{2}}_i\bigg[e^{\sum_{t=0}^{\check{\tau}(\hat{\mathscr{B}}_1)\wedge N-1}(c(X_t,\pi^{*1}(X_t),\pi^{2}_t)-\rho^*)}\psi^{*}(X_{\check{\tau}(\hat{\mathscr{B}}_1)\wedge N})\bigg]~\forall i\in {\hat{\mathscr{B}}_1}^c.
	\end{equation*}
In view of Lemma \ref{lemm 2.4}	by Fatou's lemma taking $N\rightarrow\infty$, we get
	\begin{equation}
	\psi^{*}(i)\geq E^{\pi^{*1},\pi^{2}}_i\bigg[e^{\sum_{t=0}^{\check{\tau}(\hat{\mathscr{B}}_1)-1}(c(X_t,\pi^{*1}(X_t),\pi^{2}_t)-\rho^*)}\psi^{*}(X_{\check{\tau}(\hat{\mathscr{B}}_1)})\bigg],~\forall i\in {\hat{\mathscr{B}}_1}^c.\label{eq 2.39}
	\end{equation}
	Hence,
	\begin{align}
	\psi^{*}(i)&\geq\sup_{\pi^2\in \Pi^2_{ad}} E^{\pi^{*1},\pi^{2}}_i\bigg[e^{\sum_{t=0}^{\check{\tau}(\hat{\mathscr{B}}_1)-1}(c(X_t,\pi^{*1}(X_t),\pi^{2}_t)-\rho^*)}\psi^{*}(X_{\check{\tau}(\hat{\mathscr{B}}_1)})\bigg]\nonumber\\
	&\geq \inf_{\pi^1\in\Pi^1_{ad}}\sup_{\pi^1\in \Pi^1_{ad}}E^{\pi^1,\pi^{2}}_i\bigg[e^{\sum_{t=0}^{\check{\tau}(\hat{\mathscr{B}}_1)-1}(c(X_t,\pi^1_t,\pi^{2}_t)-\rho^*)}\psi^{*}(X_{\check{\tau}(\hat{\mathscr{B}}_1)})\bigg],~\forall i\in{\hat{\mathscr{B}}_1}^c.\label{eq 2.40}
	\end{align}
	From (\ref{eq 2.42}) and (\ref{eq 2.40}), we get eq. (\ref{eq 2.31}).
	\end{proof}
Next we prove the existence of the value of the game. To this end we first perturb the cost function as follows:
\begin{itemize}
	\item 	When Assumption \ref{assm 2.2} (a) holds: Let $\alpha_3>0$, be a small number satisfying $\|{c}\|_\infty+\alpha_3<\tilde{\gamma}$. Now we define
	$\tilde{c}_n(i,u,v)={c}(i,u,v)I_{\tilde{\mathscr{D}}_n}(i)+(\|{c}\|_\infty+\alpha_3)I_{\tilde{\mathscr{D}}_n^c}$ $\forall (u,v)\in U(i)\times V(i)$, $i\in S$.
	Note $\|\tilde{c}_n\|_{\infty}<\tilde{\gamma}$, where $\|\tilde{c}_n\|_{\infty}=\sup_{(i,u,v)\in \mathcal{K}}\tilde{c}_n(i,u,v)$.
	\item When Assumption \ref{assm 2.2} (b) holds: Define
	$\displaystyle \tilde{c}_n(i,u,v)={c}(i,u,v)+\frac{1}{n}[\tilde{\ell}(i)-\sup_{(u,v)\in U(i)\times V(i)}{c}(i,u,v)]_+$ $\forall (u,v)\in U(i)\times V(i)$, $i\in S$. Since the function $\displaystyle [\tilde{\ell}(\cdot)-\sup_{(u,v)\in U(\cdot)\times V(\cdot)}{c}(\cdot,u,v)]_+$  is norm-like function,  we have  $\displaystyle \tilde{\ell}-\sup_{(u,v)\in U(\cdot)\times V(\cdot)}\tilde{c}_n(\cdot,u,v)$ is norm-like for large enough $n$.
\end{itemize}
\begin{thm}\label{theo 3.1}
Suppose that Assumptions \ref{assm 2.2}, and \ref{assm 2.3} hold. Let $(\pi^{*1},\pi^{*2})\in \Pi_{SM}^1\times \Pi_{SM}^2$ be any mini-max selector of (\ref{eq 2.30}), i.e. $(\pi^{*1},\pi^{*2})\in \Pi_{SM}^1\times \Pi_{SM}^2$ satisfies
	\begin{align}
	e^{\rho^*}\psi^{*}(i)
	&= \sup_{\nu\in \mathcal{P}(V(i))}\inf_{\mu\in \mathcal{P}(U(i))}\bigg[e^{c(i,\mu,\nu)}\sum_{j\in S}\psi^{*}(j)P(j|i,\mu,\nu)\bigg]\nonumber\\
	&=\inf_{\mu\in \mathcal{P}(U(i))}\sup_{\nu\in \mathcal{P}(V(i))}\bigg[e^{c(i,\mu,\nu)}\sum_{j\in S}\psi^{*}(j)P(j|i,\mu,\nu)\bigg]\nonumber\\
	&=\inf_{\mu\in \mathcal{P}(U(i))}\bigg[e^{c(i,\mu,\pi^{*2}(i))}\sum_{j\in S}\psi^{*}(j)P(j|i,\mu,\pi^{*2}(i))\bigg]\nonumber\\
	&=\sup_{\nu\in \mathcal{P}(V(i))}\bigg[e^{c(i,\pi^{*1}(i),\nu)}\sum_{j\in S}\psi^{*}(j)P(j|i,\pi^{*1}(i),\nu)\bigg], ~i\in S.\label{eq 2.46}
	\end{align}
Then we have
		\begin{align}
	\rho^{*}&=\inf_{i\in S}\sup_{\pi^2\in \Pi^2_{ad}}\inf_{\pi^1\in \Pi^1_{ad}}\mathscr{J}^{\pi^1,\pi^2}(i,c)=\inf_{i\in S}\inf_{\pi^1\in \Pi^1_{ad}}\sup_{\pi^2\in \Pi^2_{ad}}\mathscr{J}^{\pi^1,\pi^2}(i,c)\nonumber\\
		&=\inf_{i\in S}\inf_{\pi^1\in \Pi^1_{ad}}\mathscr{J}^{\pi^{1},\pi^{*2}}(i,c)=\inf_{i\in S}\sup_{\pi^2\in \Pi^2_{ad}}\mathscr{J}^{\pi^{*1},\pi^{2}}(i,c) = \mathscr{J}^{\pi^{*1},\pi^{*2}}(i,c).	\label{E4.2}
	\end{align}
\end{thm}
\begin{proof}
	Arguing as Lemma \ref{L2.4}, for the stationary strategy $\pi^{*1}\in  \Pi_{SM}^1$,  there exists an eigenpair $(\hat{\psi}_n,\hat{\rho}_{n})\in \mathbb{R}_{+}\times L^\infty_W$ with
	$\hat{\psi}_n>0$ satisfying
	\begin{equation}
	e^{\hat{\rho}_{n}}\hat{\psi}_n(i)=\sup_{\nu\in \mathcal{P}(B(i))}\bigg[e^{{\tilde{c}_n(i,\pi^{*1}(i),\nu)}}\sum_{j\in S}\hat{\psi}_n(j)P(j|i,\pi^{*1}(i),\nu)\bigg]\label{eq 2.48}
	\end{equation}
	such that
	\begin{equation}
	0\leq \hat{\rho}_{n}\leq\sup_{\pi^2\in \Pi^2_{ad}} \mathscr{J}^{\pi^{*1},\pi^{2}}(i,\tilde{c}_n).\label{eq 2.49}
	\end{equation}
	Also,
	\begin{equation}
	\hat{\psi}_n(i)=\sup_{\pi^2\in \Pi^2_{ad}}E^{\pi^{*1},\pi^{2}}_i\bigg[e^{\sum_{t=0}^{\check{\tau}({\hat{\mathscr{B}}_1})-1}(\tilde{c}_n(X_t,\pi^{*1}(X_t),\pi^{2}_t)-\hat{\rho}_{n})}\hat{\psi}_n(X_{\check{\tau}(\hat{\mathscr{B}}_1)})\bigg],~i\in {\hat{\mathscr{B}}_1}^c,\label{eq 2.45}
	\end{equation}
	for some finite set ${\hat{\mathscr{B}}}_1$ containing ${\hat{\mathscr{B}}}$.
	Now as in Lemma \ref{L2.4}, we have a finite set $\tilde{\mathscr{B}}_1$, depending on $n$, containing $\tilde{\mathscr{K}}$ such that the following happen:
	\begin{itemize}
		\item Under Assumption \ref{assm 2.2} (a): From (\ref{eq 2.49}), we have $\hat{\rho}_{n}\leq \|\tilde{c}_n\|_\infty$. So, from the above definition of $\tilde{c}_n$, for $i\in {\hat{\mathscr{D}}_n}^c$, we have $\tilde{c}_n(i,u,v)-\hat{\rho}_{n}\geq 0$ for all $(u,v)\in U(i)\times V(i)$. Consequently, we may take $\tilde{\mathscr{B}}_1=\hat{\mathscr{D}}_n$ such that $\tilde{c}_n(i,u,v)-\hat{\rho}_{n}\geq 0$ in $\tilde{\mathscr{B}}_1^c$ for all $(u,v)\in U(i)\times V(i)$.
		\item  Under Assumption \ref{assm 2.2} (b): since $\tilde{c}_n$ is norm-like function, we can choose suitable finite set $\tilde{\mathscr{B}}_1$ such that $(\tilde{c}_n(i,u,v)-\hat{\rho}_{n})\geq 0$ in $\tilde{\mathscr{B}}_1^c$ for all $(u,v)\in U(i)\times V(i)$.
	\end{itemize}
	From (\ref{eq 2.48}), we obtain
	\begin{equation}
	\hat{\psi}_n(i)\geq \bigg[e^{{(\tilde{c}_n(i,\pi^{*1}(i),\nu)-\hat{\rho}_{n})}}\sum_{j\in S}\hat{\psi}_n(j)P(j|i,\pi^{*1}(i),\nu)\bigg].\label{L2.4E2.53}
	\end{equation}
	By Dynkin's formula from (\ref{L2.4E2.53}), we deduce that
	\begin{equation*}
	\hat{\psi}_n(i)\geq E^{\pi^{*1},\pi^{2}}_i\bigg[e^{\sum_{t=0}^{\check{\tau}(\tilde{\mathscr{B}}_1)\wedge N-1}(\tilde{c}_n(X_t,\pi^{*1}(X_t),\pi^{2}_t)-\hat{\rho}_{n})}\hat{\psi}_n(X_{\check{\tau}(\tilde{\mathscr{B}}_1)\wedge N})\bigg].
	\end{equation*}
	Since $\tilde{c}_n(i,u,v)-\hat{\rho}_{n}\geq 0$, in $\tilde{\mathscr{B}}_1^c$, for all $(u,v)\in U(i)\times V(i)$, by Fatou lemma taking $N\rightarrow \infty$, we obtain
	\begin{align*}
	\hat{\psi}_n(i)&\geq E^{\pi^{*1},\pi^{2}}_i\bigg[e^{\sum_{t=0}^{\check{\tau}(\tilde{\mathscr{B}}_1)-1}(\tilde{c}_n(X_t,\pi^{*1}(X_t),\pi^{2}_t)-\hat{\rho}_{n})}\hat{\psi}_n(X_{\check{\tau}(\tilde{\mathscr{B}}_1)})\bigg]\geq \min_{\tilde{\mathscr{B}}_1}\hat{\psi}_n ~ \forall ~i\in \tilde{\mathscr{B}}_1^c.
	\end{align*}
	So, $\hat{\psi}_n$ has a lower bound.
	Again by Dynkin's formula from (\ref{eq 2.48}), we get
	\begin{align*}
	\hat{\psi}_n(i)&\geq  E^{\pi^{*1},\pi^{2}}_i\bigg[e^{\sum_{t=0}^{T\wedge{\tau}(\tilde{\mathscr{D}}_m)-1}(\tilde{c}_n(X_t,\pi^{*1}(X_t),\pi^{2}_t)-\hat{\rho}_{n})}\hat{\psi}_n(X_{T\wedge{\tau}(\tilde{\mathscr{D}}_m)})\bigg].
	\end{align*}
	By Fatou's lemma, taking $m\rightarrow\infty$, we obtain
	\begin{align*}
	\hat{\psi}_n(i)&\geq  E^{\pi^{*1},\pi^{2}}_i\bigg[e^{\sum_{t=0}^{T-1}(\tilde{c}_n(X_t,\pi^{*1}(X_t),\pi^{2}_t)-\hat{\rho}_{n})}\hat{\psi}_n(X_{T})\bigg]\\ &\geq (\min_{\tilde{\mathscr{B}}_1}\hat{\psi}_n)E^{\pi^{*1},\pi^{2}}_i\bigg[e^{\sum_{t=0}^{T-1}(\tilde{c}_n(X_t,\pi^{*1}(X_t),\pi^{2}_t)-\hat{\rho}_{n})}\bigg].
	\end{align*}
	So, taking logarithm both sides, dividing by $T$ and letting $T\rightarrow\infty$, we deduce that
$$\hat{\rho}_{n}\geq\mathscr{J}^{\pi^{*1},\pi^{2}}(i,\tilde{c}_n).$$
	Since $\pi^2\in \Pi^2_{ad}$ is arbitrary,
	\begin{align*}
	\hat{\rho}_{n}&\geq\sup_{\pi^2\in \Pi^2_{ad}} \mathscr{J}^{\pi^{*1},\pi^{2}}(i,\tilde{c}_n)\geq\sup_{\pi^2\in \Pi^2_{ad}} \mathscr{J}^{\pi^{*1},\pi^{2}}(i,{c}).
	\end{align*}
	Using this and (\ref{eq 2.49}), we get $\displaystyle \sup_{\pi^2\in \Pi^2_{ad}}\mathscr{J}^{\pi^{*1},\pi^{2}}(i,c)\leq\sup_{\pi^2\in \Pi^2_{ad}} \mathscr{J}^{\pi^{*1},\pi^{2}}(i,\tilde{c}_n)=\hat{\rho}_{n}$ for all $n$. Now arguing as before, it follows that $\hat{\psi}_n\leq \mathcal{W}$ and it touches $\mathcal{W}$ (by suitable scaling). Also, we note  from the definition of $\tilde{c}_n$ that $\hat{\rho}_{n}$ is a monotone decreasing sequence bounded below. Thus, using diagonalization arguments, there exists subsequence (denoting the the same sequence) and a pair $(\hat{\rho},\hat{\psi})$ such that $\hat{\rho}_{n}\rightarrow \hat{\rho}$ and $\hat{\psi}_n\rightarrow\hat{\psi}$ as $n\rightarrow\infty$.
	Now applying similar arguments as Lemma \ref{L2.4}, taking $n\rightarrow\infty$ in (\ref{eq 2.48}) we get
	\begin{equation}
	e^{{\hat{\rho}}}\hat{\psi}(i)=\sup_{\nu\in \mathcal{P}(V(i))}\bigg[e^{{{c}(i,\pi^{*1}(i),\nu)}}\sum_{j\in S}\hat{\psi}(j)P(j|i,\pi^{*1}(i),\nu)\bigg].\label{eq 2.50}
	\end{equation}
	We also have $\displaystyle \lim_{n\rightarrow \infty}\hat{\rho}_{n}=\hat{\rho}\geq\sup_{\pi^2\in \Pi^2} \mathscr{J}^{\pi^{*1},\pi^{2}}(i,c)\geq \rho^*$. So, we have to prove $\hat{\rho}=\rho^*$.
	By continuity-compactness assumptions, there exists $\hat{\pi}^{*2}$ such that (\ref{eq 2.50}) can be rewritten as
	\begin{align}
	e^{{\hat{\rho}}}\hat{\psi}(i)=\bigg[e^{{{c}(i,{\pi}^{*1}(i),\hat{\pi }^{*2}(i))}}\sum_{j\in S}\hat{\psi}(j)P(j|i,{\pi}^{*1}(i),\hat{\pi}^{*2}(i))\bigg].\label{E2.50}
	\end{align}
By Dynkin's formula, for some $\hat{\mathscr{B}}_2$ containing $\hat{\mathscr{B}}$, we have
	\begin{align}
\hat{\psi}(i)=E^{{\pi}^{*1},\hat{\pi}^{*2}}_i\bigg[e^{\sum_{t=0}^{\check{\tau}({\hat{\mathscr{B}}_2})\wedge N-1}({c}(X_t,{\pi}^{*1}(X_t),\hat{\pi }^{*2}(X_t))-\hat{\rho})}\hat{\psi}(X_{(\check{\tau}({\hat{\mathscr{B}}_2})\wedge N)})\bigg],~\forall i\in {{\hat{\mathscr{B}}_2}}^c.\label{E3.9}
\end{align}
Using the estimates of Lemma \ref{lemm 2.4} and the dominated convergence theorem, taking $N\rightarrow \infty$ in (\ref{E3.9}), we have
	\begin{align}
	\hat{\psi}(i)=E^{{\pi}^{*1},\hat{\pi }^{*2}}_i\bigg[e^{\sum_{t=0}^{\check{\tau}({\hat{\mathscr{B}}_2})-1}({c}(X_t,{\pi}^{*1}(X_t),\hat{\pi }^{*2}(X_t))-\hat{\rho})}\hat{\psi}(X_{\check{\tau}({\hat{\mathscr{B}}_2})})\bigg],~\forall i\in {{\hat{\mathscr{B}}_2}}^c.\label{eq 2.51}
	\end{align}
	Since $\hat{\rho}\geq \rho^*$, from (\ref{eq 2.31}) we have
	\begin{align}
	\psi^{*}(i)\geq E^{{\pi}^{*1},\hat{\pi}^{*2}}_i\bigg[e^{\sum_{t=0}^{\check{\tau}({{\hat{\mathscr{B}}}_2})-1}({c}(X_t,{\pi}^{*1}(X_t),\hat{\pi}^{*2}(X_t))-\hat{\rho})}\psi^{*}(X_{\check{\tau}({\hat{\mathscr{B}}_2})})\bigg]~\forall i\in {{\hat{\mathscr{B}}}_2}^c.\label{eq 2.52}
	\end{align}
	Hence from (\ref{eq 2.51}) and (\ref{eq 2.52}), we get
	\begin{align}
	\psi^{*}(i)-\hat{k}_1\hat{\psi}(i)\geq E^{{\pi}^{*1},\hat{\pi}^{*2}}_i\bigg[e^{\sum_{t=0}^{\check{\tau}({\mathscr{B}}_2)-1}({c}(X_t,{\pi}^{*1}(X_t),\hat{\pi}^{*2}(X_t))-\hat{\rho})}(\psi^{*}-\hat{k}_1\hat{\psi})(X_{\check{\tau}(\hat{\mathscr{B}}_2)})\bigg]~\forall i\in {\hat{\mathscr{B}}_2}^c.\label{eq 2.53}
	\end{align}
Let $\hat{k}_1 = \displaystyle \min_{\hat{\mathscr{B}}_2}\frac{\psi^{*}}{\hat{\psi}}$, thus we have $(\psi^{*}-\hat{k}_1\hat{\psi}) \geq 0$ in $\hat{\mathscr{B}}_2$ and for some $\hat{i}_0\in \hat{\mathscr{B}}_2,$ \,$(\psi^{*}-\hat{k}_1\hat{\psi})(\hat{i}_0) = 0$. Therefore, from (\ref{eq 2.53}), we obtain that $(\psi^{*}-\hat{k}_1\hat{\psi}) \geq 0$ in $S$. Now since $\hat{\rho}\geq \rho^*$, from (\ref{eq 2.46}) and (\ref{E2.50}), we deduce that
	\begin{align*}
	e^{\hat{\rho}}(\psi^{*} - \hat{k}_1\hat{\psi})(\hat{i}_0) \geq \biggl[e^{ c(\hat{i}_0,{\pi}^{*1}(i_0),\hat{\pi}^{*2}(i_0))}\sum_{j\in S}(\psi^{*} -\hat{k}_1\hat{\psi})(j)P(j|\hat{i}_0 ,\pi^{*1}(\hat{i}_0),\hat{\pi}^{*2}(\hat{i}_0)) \biggr].	
	\end{align*} This implies that
	\begin{align*}
	0=\sum_{j\in S}(\psi^{*} -\hat{k}_1\hat{\psi})(j)P(j|\hat{i}_0 ,\pi^{*1}(\hat{i}_0),\hat{\pi}^{*2}(\hat{i}_0)).	
	\end{align*}
Thus, in view of irreducibility property of the Markov chain under stationary Markov strategies, it follows that $\psi^{*}=\hat{k}_1\hat{\psi}$ in $S$. Hence from (\ref{eq 2.46}) and (\ref{eq 2.50}), it is easy to see that $\hat{\rho}=\rho^* = \mathscr{J}^{\pi^{*1},\pi^{*2}}(i,c)$ for all $i \in S$. Therefore, we obtain
	\begin{align}
	\rho^{*}&=\inf_{i\in S}\sup_{\pi^2\in \Pi^2_{ad}}\inf_{\pi^1\in \Pi^1_{ad}}\mathscr{J}^{\pi^1,\pi^2}(i,c)=\inf_{i\in S}\inf_{\pi^1\in \Pi^1_{ad}}\sup_{\pi^2\in \Pi^2_{ad}}\mathscr{J}^{\pi^1,\pi^2}(i,c)=\inf_{i\in S}\sup_{\pi^2\in \Pi^2_{ad}}\mathscr{J}^{\pi^{*1},\pi^{2}}(i,c) = \mathscr{J}^{\pi^{*1},\pi^{*2}}(i,c).
	\label{E3.12}
	\end{align}
Now arguing as in [\cite{BP}, Lemma 2.6], it follows that for $\pi^{*2}\in \Pi_{SM}^2$, there exists $({\psi}^{'},{\rho}^{'})\in  L^\infty_{\mathcal{W}} \times\mathbb{R}_{+}$, ${\psi}^{'}>0$ satisfying
	\begin{equation}
	e^{{\rho}^{'}}{\psi}^{'}(i)=\inf_{\mu\in \mathcal{P}(U(i))}\bigg[e^{{c}(i,\mu,\pi^{*2}(i))}\sum_{j\in S}{\psi}^{'}(j)P(j|i,\mu,\pi^{*2}(i))\bigg],\label{eq 3.12}
	\end{equation}
	with
	\begin{equation}
	{\rho}^{'}= \inf_{i\in S}\inf_{\pi^1\in\Pi^1_{ad}}\mathscr{J}^{\pi^{1},\pi^{*2}}(i,c).\label{eq 3.13}
	\end{equation}
Thus, we have
	\begin{equation}
	\rho^{'} = \inf_{i\in S}\inf_{\pi^1\in\Pi^1_{ad}}\mathscr{J}^{\pi^{1},\pi^{*2}}(i,c)\leq \inf_{i\in S}\sup_{\pi^2\in \Pi^2_{ad}}\inf_{\pi^1\in\Pi^1_{ad}}\mathscr{J}^{\pi^{1},\pi^{2}}(i,c) = \rho^*.\label{eq E3.14E}
\end{equation}
For any minimizing selector $\tilde{\pi}^{*1}$ of (\ref{eq 3.12}) we obtain
	\begin{equation}
	e^{{\rho}^{'}}{\psi}^{'}(i)=\bigg[e^{{c}(i,\tilde{\pi}^{*1}(i),\pi^{*2}(i))}\sum_{j\in S}{\psi}^{'}(j)P(j|i,\tilde{\pi}^{*1}(i),\pi^{*2}(i))\bigg].\label{eq 3.15}
	\end{equation}
		Also, arguing as in Lemma \ref{L2.4}, for some finite set $\hat{\mathscr{B}}_3 \supset \hat{\mathscr{B}}$\,, we deduce that
	\begin{equation}
	{\psi}^{'}(i)= E^{\tilde{\pi}^{*1},\pi^{*2}}_i\bigg[e^{\sum_{t=0}^{\check{\tau}({\hat{\mathscr{B}}_3})-1}({c}(X_t,\tilde{\pi}^{*1}(X_t)),\pi^{*2}(X_t))-{\rho}^{'})dt}{\psi}^{'}(X_{\check{\tau}(\hat{\mathscr{B}}_3)})\bigg],~i\in {\hat{\mathscr{B}}_3}^c.\label{eq 3.17}
	\end{equation}
From (\ref{eq 2.46}), we have
	\begin{align}
	e^{\rho^*}\psi^{*}(i)
	\leq\bigg[e^{c(i,\tilde{\pi}^{*1}(i),\pi^{*2}(i))}\sum_{j\in S}\psi^{*}(j)P(j|i,\tilde{\pi}^{*1}(i),\pi^{*2}(i))\bigg], ~i\in S\,.\label{eq 3.16}
	\end{align}
Also, from (\ref{eq 2.31}), it follows that
	\begin{align}
	\psi^{*}(i)\leq
	E^{\tilde{\pi}^{*1},\pi^{*2}}_i\bigg[e^{\sum_{t=0}^{\check{\tau}(\mathscr{B}_3)-1}(c(X_t,\tilde{\pi}^{*1}(X_t),\pi^{*2}(X_t))-\rho^*)dt}\psi^{*}(X_{\check{\tau}(\mathscr{B}_3)})\bigg]~\forall i\in {\hat{\mathscr{B}}_3}^c\,.\label{eq 3.18}
	\end{align}
	Therefore, by analogous arguments as above, using irreducibility property of the Markov chain we get $\psi^{'}=\hat{k}_2 \psi^*$, for some positive constant $\hat{k}_2$. Thus, from (\ref{eq 2.46}) and (\ref{eq 3.12}), it follows that
	\begin{equation}
	\rho^*=\rho^{'}.\label{eq 3.19}
	\end{equation}
	Hence, by (\ref{E3.12}) and (\ref{eq 3.19}), we obtain (\ref{E4.2}). This completes the proof of the theorem.
	\end{proof}
\begin{remark}
As in the proof of Theorem~\ref{theo 3.1}, exploiting the stochastic representation of $\psi^{*}$ and irreducibility of the Markov chain, it is easy to see that $\psi^{*}$ is unique solution of (\ref{eq 2.30}) (upto a multiplicative constant). If we set $\psi^{*}(i_0) = 1$, where $i_0$ is the reference state in Assumption 2.2(ii), then  $\psi^{*}$ is unique.
\end{remark}
Next we prove the converse of the above theorem.
\begin{thm}\label{theo 3.3}
Suppose Assumptions \ref{assm 2.2}, and \ref{assm 2.3} hold. Suppose there exists a saddle point equilibrium $(\hat{\pi}^{*1},\hat{\pi}^{*2})\in \Pi_{SM}^1\times \Pi_{SM}^2$\,, i.e., for all $i\in S$\,,
	\begin{align}
 &\mathscr{J}^{\hat{\pi}^{*1},\hat{\pi}^{*2}}(i,c)\leq \mathscr{J}^{\pi^1,\hat{\pi}^{*2}}(i,c),~\text{for all}~\pi^1\in \Pi^1_{ad},\nonumber\\
 &\mathscr{J}^{\hat{\pi}^{*1},\hat{\pi}^{*2}}(i,c)\geq 	\mathscr{J}^{\hat{\pi}^{*1},\pi^{2}}(i,c) ,~\text{for all}~\pi^2\in \Pi^2_{ad}.\label{ET3.2}
	\end{align}
	Then $(\hat{\pi}^{*1},\hat{\pi}^{*2})$ is a mini-max selector of (\ref{eq 2.30}).
\end{thm}
\begin{proof}
By Theorem \ref{theo 3.1} and (\ref{ET3.2}), we have
	\begin{align*}
	\rho^* = \inf_{\pi^1\in\Pi^1_{ad}}\sup_{\pi^2\in\Pi^2_{ad}}\mathscr{J}^{\pi^{1},\pi^{2}}(i,c) & \leq \sup_{\pi^2\in\Pi^2_{ad}}\mathscr{J}^{\hat{\pi}^{*1},\pi^{2}}(i,c) \leq \mathscr{J}^{\hat{\pi}^{*1},\hat{\pi}^{*2}}(i,c) \\
	&\leq \inf_{\pi^1\in\Pi^1_{ad}}\mathscr{J}^{\pi^{1},\hat{\pi}^{*2}}(i,c) \leq \sup_{\pi^2\in\Pi_{ad}^2} \inf_{\pi^1\in\Pi^1_{ad}} \mathscr{J}^{\pi^{1},\pi^{2}}(i,c) = \rho^*\,.
	\end{align*}
	This implies that $\displaystyle \rho^* = \mathscr{J}^{\hat{\pi}^{*1},\hat{\pi}^{*2}}(i,c) =\sup_{\pi^2\in\Pi^2_{ad}} \mathscr{J}^{\hat{\pi}^{*1},{\pi}^{2}}(i,c) = \inf_{\pi^1\in\Pi^1_{ad}}\mathscr{J}^{\pi^{1},\hat{\pi}^{*2}}(i,c)$\,. Now arguing as in Lemma \ref{L2.4} and Theorem \ref{theo 3.1}, it follows that for $\hat{\pi}^{*2}\in\Pi^2_{SM}$ there exists $(\rho^{\hat{\pi}^{*2}}, \psi_{\hat{\pi}^{*2}}^{*})\in \mathbb{R}_+\times L^\infty_{\mathcal{W}}$ with $\psi_{\hat{\pi}^{*2}}^{*} > 0$ such that
	\begin{align}
	e^{\rho^{\hat{\pi}^{*2}}}\psi_{\hat{\pi}^{*2}}^{*}(i)=\inf_{\mu\in \mathcal{P}(U(i))}\bigg[e^{c(i,\mu,\hat{\pi}^{*2}(i))}\sum_{j\in S}\psi_{\hat{\pi}^{*2}}^{*}(j)P(j|i,\mu,\hat{\pi}^{*2}(i))\bigg],\label{T3.3A}
	\end{align}
	and $\displaystyle\rho^{\hat{\pi}^{*2}}=\inf_{\pi^1\in\Pi^1}\mathscr{J}^{\pi^1,\hat{\pi}^{*2}}(i,c) = \rho^*$. Thus for $\pi^{*1}$ as in (\ref{eq 2.46}), we have
	\begin{align}
	e^{\rho^{\hat{\pi}^{*2}}}\psi_{\hat{\pi}^{*2}}^{*}(i) \leq \bigg[e^{c(i,\pi^{*1}(i),\hat{\pi}^{*2}(i))}\sum_{j\in S}\psi_{\hat{\pi}^{*2}}^{*}(j)P(j|i,\pi^{*1}(i),\hat{\pi}^{*2}(i))\bigg].\label{T3.3B}
	\end{align} Arguing as in Lemma \ref{L2.4}, for some finite set $\hat{\mathscr{B}}_4\supset \hat{\mathscr{B}}$ it follows that
	\begin{align}
	\psi_{\hat{\pi}^{*2}}^{*}(i) \leq  E^{{\pi}^{*1},\hat{\pi}^{*2}}_i\bigg[e^{\sum_{t=0}^{\check{\tau}(\mathscr{B}_4)-1}(c(X_t,\pi^{*1}(X_t),\hat{\pi}^{*2}(X_t))-\rho^*)}\psi_{\hat{\pi}^{*2}}^{*}(X_{\check{\tau}(\hat{\mathscr{B}}_4)})\bigg]~\forall i\in{ \hat{\mathscr{B}}_4}^c.\label{T3.3C}
	\end{align}
Also, from (\ref{eq 2.46}), we deduce that
\begin{align}
	e^{\rho^*}\psi^{*}(i) \geq \bigg[e^{c(i,\pi^{*1}(i),\hat{\pi}^{*2}(i))}\sum_{j\in S}\psi^{*}(j)P(j|i,\pi^{*1}(i),\hat{\pi}^{*2}(i))\bigg].\label{eq 3.34}
	\end{align}
	 By Dynkin's formula and Fatou's lemma (as in Lemma \ref{L2.4}), we obtain
	\begin{align}
\psi^{*}(i) \geq  E^{{\pi}^{*1},\hat{\pi}^{*2}}_i\bigg[e^{\sum_{t=0}^{\check{\tau}(\mathscr{B}_4)-1}(c(X_t,\pi^{*1}(X_t),\hat{\pi}^{*2}(X_t))-\rho^*)}\psi^{*}(X_{\check{\tau}(\hat{\mathscr{B}}_4)})\bigg]~\forall i\in{ \hat{\mathscr{B}}_4}^c.\label{eq 3.33}
\end{align}
	Now in view of (\ref{T3.3C}) and (\ref{eq 3.33}) and applying the same technique as before (as in Theorem~\ref{theo 3.1}), it follows that
	$\psi^{*}=\hat{k}_2 \psi_{\hat{\pi}^{*2}}^{*}$, for some constant $\hat{k}_2>0$. Hence from (\ref{eq 2.46}) and (\ref{T3.3A}), it is easy to see that $\hat{\pi}^{*2}$ is an outer maximizing selector of (\ref{eq 2.30}). Similarly, one can show that $\hat{\pi}^{*1}$ is an outer minimizing selector of (\ref{eq 2.30}). This completes the proof.
\end{proof}
Now we are ready to prove Theorem \ref{theo 2.2}.\\
\textbf{Proof of Theorem \ref{theo 2.2}:}
\begin{proof} Existence of an eigenpair $(\rho^*,\psi^{*})$ of eq. (\ref{T2.21}) follows from Lemma \ref{L2.4}. From Theorem \ref{theo 3.1}, we have Theorem \ref{theo 2.2}(i) and Theorem \ref{theo 2.2}(ii). Theorem \ref{theo 2.2}(iii) follows from Theorem \ref{theo 3.3}. This completes the proof.
	\end{proof}
\begin{remark}\label{R1} We can replace Assumption \ref{assm 2.3} (ii) by other similar assumption. For example, if the killed process communicates with every state in $\tilde{\mathscr{D}}_n$ from $i_0$ before leaving the domain $\tilde{\mathscr{D}}_n$, for large $n$, then our method applies. More precisely, we can replace Assumption \ref{assm 2.3} (ii) with the following: for all large n, we have \begin{equation*}
\inf_{(\pi^1,\pi^2)\in \Pi_{SM}^1 \times \Pi_{SM}^2}\mathbb{P}^{\pi^1,\pi^2}_{i_0}(\check{\tau}_j<{\tau}(\tilde{\mathscr{D}}_n))>0~\text{ for all}~j\in \tilde{\mathscr{D}}_n\backslash\{i_0\},
		\end{equation*} where $\check{\tau}_j$ denotes the hitting time to $j$. In other words, for every $\tilde{\mathscr{D}}_n$, $({\pi}^1,{\pi}^2)\in \Pi_{SM}^1\times \Pi_{SM}^2$ and for every $j\in \tilde{\mathscr{D}}_n\backslash \{i_0\}$, if there exists distinct $i_1,i_2,\cdots,i_m\in \tilde{\mathscr{D}}_n\backslash \{i_0\}$ satisfying $$P(i_1|i_0,\pi^1(i_0),\pi^2(i_0))P(i_2|i_1,\pi^1(i_1),\pi^2(i_1))\cdots P(j|i_m,\pi^1(i_m),\pi^2(i_m))>0,$$ then we get $\psi_{n}(i_0)>0$ in $\tilde{\mathscr{D}}_n$ (see Lemma ~\ref{lemm 2.3}). Also, the conclusion of Theorem~\ref{theo 3.1} holds.
\end{remark}
 \section{ Example}
We present here an illustrative example  in which all our assumptions hold, and  the cost function is nonnegative and unbounded.
\begin{example}
Consider a controlled birth-and-death system in
which the state variable stands for the total population size at time $t\geq 0$. Thus, the state space can be represented
by $ S:=\{0,1,2,\cdots\}$. Suppose that there are two players, player 1 and player 2, and they can control birth  and death, respectively. Depending on the number of population in the
system,  player 1 can modify the number of births  by choosing some action $u$, from the set $U(i)=[\delta,L_1]$. But this action
 results in a cost  given by $\tilde{c}_1(i,u)\geq 0$ (or a reward $\tilde{c}_1(i,u)\leq 0$), if $i$ is the state of the system. On the other hand, player 2 can modify the number of deaths
 by choosing some action $v$ from the set $V(i)=[\delta,L_1]$. The action of player 2 incurs a cost given by $\tilde{c}_2(i,v)\geq 0$ (or a reward $\tilde{c}_2(i,v)\leq 0$).  Also, in addition, assume that player 1 `owns' the system and he/she gets a  cost $r(i):=\hat{p}\cdot i$ for each unit of time during which the system remains in the state $i\in S$, where $\hat{p} >0$ is a fixed cost per population.


We next formulate this model as a discrete-time Markov game. The corresponding transition stochastic kernel ${P}(j|i,u,v)$ and reward ${{c}}(i,u,v)$ for player 1 are given as follows: for $(0,u,v)\in \mathcal{K}$ ($\mathcal{K}$ as in the game model (\ref{eq 2.1})).
	\begin{equation}
	\sum_{j\in S}{P}(j|0,u,v)=1,~\text{and}~{P}(j|0,u,v) = e^{-\frac{j^2}{3}-3}~\forall~j\geq 1.\label{eq 4.1}
	\end{equation}
		Similarly, for $(1,u,v)\in \mathcal{K}$,
	\begin{align}
	{{P}}(j|1,u,v)= \left\{ \begin{array}{lll}\displaystyle{}&1-\frac{3e^{-2}v}{2(L_1+L_2)},~~\text{if}~j=0\nonumber\\
	&\frac{e^{-2}v}{2(L_1+L_2)},~~\text{if}~j=1\nonumber\\
	&\frac{e^{-2}v}{2(L_1+L_2)}~~\text{if}~j=2\nonumber\\
	& \frac{e^{-2}v}{2(L_1+L_2)},~~\text{if}~j=3\nonumber\\
	&0,~~~~~~~~~~~\text{otherwise}.\displaystyle{}
	\end{array}\right.
	\end{align}
	Also, for $(i,u,v)\in \mathcal{K}$ with $i\geq 2$,
\begin{align}
{P}(j|i,u,v)= \left\{ \begin{array}{lll}\displaystyle{}
&\frac{ue^{-i}}{2(L_1+L_2)},
~\text{if}~j=i-1\nonumber\\
&\frac{ue^{-i}+ve^{-2i}}{2(L_1+L_2)}~\text{if}~j=i\nonumber\\
& \frac{ve^{-2i}}{2(L_1+L_2)},~\text{if}~j=i+1\nonumber\\
&1-\frac{2(ue^{-i}+ve^{-2i})}{2(L_1+L_2)},\text{if}~j=0\nonumber\\
&0,~~~~~\text{otherwise}.\displaystyle{}
\end{array}\right.
\end{align}
\begin{align}
{c}(i,u,v):=\hat{p}\cdot i+\tilde{c}_1(i,u)-\tilde{c}_2(i,v)~\text{ for }~(i,u,v)\in \mathcal{K}.\label{eq 4.2}
\end{align}
	We make the following assumptions to ensure the existence of  a pair of optimal strategies.
	\begin{enumerate}
		\item [(I)] The functions $\tilde{c}_1(i,u)$, and $\tilde{c}_2(i,v)$ are continuous with their respective variables for each fixed $i\in S$.
		\item [(II)] Suppose that $\hat{p}\cdot i+\tilde{c}_1(i,u)-\tilde{c}_2(i,v)\geq 0$ for $(i,u,v)\in \mathcal{K}$ and $\hat{p} < \frac{1}{6}$.  Also, assume that $f$, is a norm-like function, where $f(i):=\displaystyle{\min_{(u,v)\in U(i)\times V(i)}}[\tilde{c}_2(i,v)-\tilde{c}_1(i,u)]$ for all $i\in S$.
	\end{enumerate}
\end{example}

\begin{proposition}\label{Prop 4.1}
	Under conditions (I)-(II), the above controlled system satisfies the Assumptions  \ref{assm 2.2} and \ref{assm 2.3}. Hence by Theorem \ref{theo 2.2}, there exists a saddle-point equilibrium for this controlled model.
\end{proposition}
\begin{proof}
	Consider the Lyapunov function $\mathcal{W}(i):=e^{\frac{i^2}{6}+1}$ for $i\in S$. Then $\mathcal{W}(i)\geq 1$ for all $i\in S$. Now for each $i\geq 2$, and $(u,v)\in U(i)\times V(i)$, we have
	\interdisplaylinepenalty=0
	\begin{align}
	&\sum_{j\in S} {P}(j|i,u,v)\mathcal{W}(j)\nonumber\\
	&={P}(i-1|i,u,v)\mathcal{W}(i-1)+{P}(i|i,u,v)\mathcal{W}(i)+{P}(i+1|i,u,v)\mathcal{W}(i+1)+{P}(0|i,u,v)\mathcal{W}(0)\nonumber\\
	&= \frac{1}{2(L_1+L_2)} \biggl[ ue^{-i} e^{\frac{(i-1)^2}{6}+1}+e^{\frac{i^2}{6}+1}\bigg(ue^{-i}+ve^{-2i}\bigg)+ve^{-2i} e^{\frac{(i+1)^2}{6}+1} \biggr] +e\biggl(1-\frac{2(ue^{-i}+ve^{-2i})}{2(L_1+L_2)}\biggr)\nonumber\\
	&=e^{\frac{i^2}{6}+1}\biggl[\frac{ue^{-i}}{2(L_1+L_2)}e^{-\frac{i}{3}+\frac{1}{6}}+\bigg(\frac{ue^{-i}+ve^{-2i}}{2(L_1+L_2)}\bigg)+\frac{ve^{-2i}}{2(L_1+L_2)}e^{\frac{i}{3}+\frac{1}{6}}+e^{-\frac{i^2}{6}}\biggl(1-\frac{2(ue^{-i}+ve^{-2i})}{2(L_1+L_2)}\biggr)\biggr]\nonumber\\
	&\leq e^{\frac{i^2}{6}+1}e^{-\frac{i}{3}+\frac{1}{6}}\biggl[\frac{ue^{-i}}{2(L_1+L_2)}+\frac{u+v}{2(L_1+L_2)}+\frac{v}{2(L_1+L_2)}+\biggl(1-\frac{2(ue^{-i}+ve^{-2i})}{2(L_1+L_2)}\biggr)\biggr]\nonumber\\
	&\leq 4e^{\frac{i^2}{6}+1}e^{-\frac{i}{3}+\frac{1}{6}}\nonumber\\
	&\leq e^{(\frac{i^2}{6}+1)-\frac{1}{3}(i+3)+4}\nonumber\\
	&=\mathcal{W}(i)e^{-\frac{1}{6}(i+3)-\frac{1}{6}(i+3)+4} \nonumber\\
	&\leq e^{-\frac{1}{6}(i+3)+4I_{\mathscr{M}}(i)}\mathcal{W}(i) \leq e^{-\frac{1}{6}(i+3)}\mathcal{W}(i)+\max_{j\in \mathscr{M}}\mathcal{W}(j)e^4 I_{\mathscr{M}}(i)\leq e^{-\tilde{\ell}(i)}\mathcal{W}(i)+\tilde{C} I_{\mathscr{M}}(i),\label{eq 4.3}
	\end{align}
	where $\tilde{\ell}(i)=\frac{1}{6}(i+3)$, $\mathscr{M}:=\{i:4-\frac{1}{6}(i+3)>0\}$, and $\tilde{C}=\max\{\max_{j\in \mathscr{M}}\mathcal{W}(j)e^4,e^{-2}\sum_{i\geq 1}(e^{-\frac{i^2}{6}}-e^{-\frac{i^2}{3}})+e\}$. It is clear that $0,1\in \mathscr{M}$.
	Also, we have
	\begin{align}
	\sum_{j\in S}{P}(j|0,u,v)W(j) =e {P}(0|0,u,v)+ \sum_{j\geq 1}e^{-2}e^{-\frac{j^2}{6}}	\leq \tilde{C} I_{\mathscr{M}}(i).\label{eq 4.4}
	\end{align}	
	By similar arguments as in (\ref{eq 4.3}), we have
		\begin{align}
	&\sum_{j\in S} {{P}}(j|1,u,v)\mathcal{W}(j)\nonumber\\
	&={{P}}(0|1,u,v)\mathcal{W}(0)+{{P}}(1|1,u,v)\mathcal{W}(1)+{{P}}(2|1,u,v)\mathcal{W}(2)+{{P}}(3|1,u,v)\mathcal{W}(3)\nonumber\\
	&=e\biggl(1-\frac{3e^{-2}v}{2(L_1+L_2)}\biggr)
	+e^{\frac{1}{6}+1}\biggl(\frac{e^{-2}v}{2(L_1+L_2)}\biggr)+e^{\frac{4}{6}+1}\bigg(\frac{e^{-2}v}{2(L_1+L_2)}\bigg)+e^{\frac{9}{6}+1}\biggl(\frac{e^{-2}v}{2(L_1+L_2)}\biggr)\nonumber\\
	&\leq e^{-\frac{1}{6}(1+3)+4I_{\mathscr{M}}(1)}\mathcal{W}(1) \leq e^{-\frac{2}{3}}\mathcal{W}(1)+\max_{j\in \mathscr{M}}\mathcal{W}(j)e^4 I_{\mathscr{M}}(1)\leq e^{-\tilde{\ell}(1)}\mathcal{W}(1)+\tilde{C} I_{\mathscr{M}}(1).\label{E4.5}
	\end{align}
	Now
		\begin{align}
	\tilde{\ell}(i)-\max_{(u,v)\in U(i)\times V(i)}{c}(i,u,v)=\frac{1}{2}+(\frac{1}{6}-\hat{p})i+\min_{(u,v)\in U(\cdot)\times V(\cdot)}[\tilde{c}_2(i,v)-\tilde{c}_1(i,u)].\label{eq 4.7}
	\end{align}
	We see from condition (II) and (\ref{eq 4.7}) that $\displaystyle \tilde{\ell}(i)-\sup_{(u,v)\in U(i)\times V(i)}{{c}}(i,u,v)$ is norm-like function.
	So, by condition (II), equations (\ref{eq 4.3}), (\ref{eq 4.4}), (\ref{E4.5}), and (\ref{eq 4.7}), Assumption \ref{assm 2.2} is satisfied.
	 Now, by (\ref{eq 4.3}), (\ref{eq 4.4}), and (\ref{E4.5}), Assumption \ref{assm 2.3} (iii) is verified. Also, by the above construction of probability kernel, (\ref{eq 4.2}), and condition (I), $P(\cdot|i,u,v)$ and $c(i,u,v)$ are continuous in $(u,v)\in U(i)\times V(i)$ for all $i,j\in S$.
	Hence by Theorem \ref{theo 2.2}, it follows that there  exists a saddle-point equilibrium for this controlled model.
\end{proof}
	\begin{remark} Note that, in view of Remark \ref{R1}, one can relax the condition (\ref{eq 4.1}).
\end{remark}

\section{Conclusion}
We have studied a risk-sensitive zero-sum stochastic game with ergodic cost criterion on a countable state space. Under
certain assumptions we have established the existence of a saddle point equilibrium and have characterized the same. Instead
of employing the traditional vanishing discount asymptotics, we have pursued a direct approach involving the principal eigenpair of the
corresponding Shapley equation. It would be interesting to study the same problem on  a general state space.

  	\bibliographystyle{elsarticle-num}

	 \nocite{*}
	\bibliographystyle{plain}

\begin{thebibliography}{00}
  		
  			\bibitem{A1} A. ARAPOSTATHIS, \textit{A counterexample to a nonlinear version of the Kreın-Rutman theorem by R. Mahadevan}, Nonlinear Anal., 171 (2018), pp. 170-176.
  		
  	
%
%
  		\bibitem{BM} S. BALAJI AND S. P. MEYN, \textit{Multiplicative ergodicity and large deviations for an irreducible Markov chain}, Stochastic processes and their applications, 90 (2000), pp. 123-144.
  		
  		\bibitem{BG} A. BASU AND M. K. GHOSH, \textit{Zero-sum risk-sensitive stochastic games on a countable state space}, Stochastic processes and their applications, 124 (2014), pp. 961-983.
  		
	\bibitem{BG1}  A. BASU AND M. K. GHOSH, \textit{Nonzero-sum risk-sensitive stochastic games on a countable state space}, Math. of Oper. Res., 43 (2018), pp. 516-532.
	
  		\bibitem{BR1} N. BAUERLE AND U. RIEDER, \textit{More risk-sensitive Markov decision processes}, Math. Oper. Res., 39 (2014), pp. 105-120.
  		
  		\bibitem{BR} N. BAUERLE AND U. RIEDER, \textit{Zero-sum risk-sensitive stochastic games}, Stochastic processes and their applications, 127 (2017), pp. 622-642.
  		
	\bibitem{BS} D. P. BERTSEKAS AND S. E. SHEREVE, \textit{Stochastic Optimal Control: The Discrete-Time Case}, Athena Scientific, Belmon, (1996).
	


  		\bibitem{BP} A. BISWAS AND S. PRADHAN, \textit{Ergodic risk-sensitive control of Markov processes on countable state space revisited}, ArXiv e-prints 2104.04825 (2021), available at https://arxiv.org/abs/2104.04825.
  		
  			
  			\bibitem{BG3} V. S. BORKAR AND S. P. MEYN, \textit{Risk-sensitive optimal control for Markov decision processes with monotone cost}, Math. Oper. Res., 27 (2002), pp. 192-209.
  			
  			
  		
 		\bibitem{CH} R. CAVAZOS-CADENA AND D. HERNANDEZ-HERNANDEZ, \textit{A characterization of the optimal risk-sensitive average cost infinite controlled Markov chains}, Ann. Appl. Probab. 15 (2005), pp. 175-212.

 		
  	
 	

	\bibitem{Fan} K. FAN, \textit{Minimax Theorems}, Proc. Natl. Acad. Sci., USA, 39 (1953), pp. 42-47.

	
  		\bibitem{GS2} M. K. GHOSH AND S. SAHA, \textit{Risk-sensitive control of continuous-time Markov chains}, Stochastic, 86 (2014), pp. 655-675.
  		
  		\bibitem{GKP} M. K. GHOSH, K. S. KUMAR AND C. PAL, \textit{Zero-sum risk-sensitive stochastic games for continuous-time Markov chains}, Stoch. Anal. Appl., 34 (2016), pp. 835-851.
  		

  		
  		\bibitem{GH1} X. P. GUO AND O. HERNANDEZ-LERMA, \textit{Zero-sum games for continuous-time Markov chains with unbounded transition and average payoff rates}, J. Appl. Probab., 40 (2003), pp. 327-345.
  		

  		\bibitem{GH2} X. GUO AND  O. HERNANDEZ-LERMA, \textit{Nonzero-sum games for continuous-time Markov chains with unbounded discounted payoffs}, J. Appl. Probab., 42 (2005) pp. 303-320.
  		

  		\bibitem{GH4} X. P. GUO AND O. HERNANDEZ-LERMA, \textit{Continuous-Time Markov decision processes}, Stochastic Modelling and Applied Probability, Springer-Verlag, Berlin, (62) (2009).
  		

  		
	\bibitem{GL}	X. GUO AND Z. W. LIAO, \textit{Risk-sensitive discounted continuous-time Markov decision processes with unbounded rates}, SIAM J. Control Optim., 57 (2019), pp. 3857-3883.
	

  		\bibitem{GZ1} X. GUO AND J. ZHANG, \textit{Risk-sensitive continuous-time Markov decision processes with unbounded rates and Borel spaces}, Discrete Event Dyn. Syst., 29 (2019), pp. 272-288.
  		


  		\bibitem{H} O. HERNANDEZ-LERMA, \textit{Adaptive Markov control processes}, Springer-Verlag, New York, (1989).
  		
  		\bibitem{HL}  O. HERNANDEZ-LERMA AND J. LASSERRE, \textit{ Further topics on discrete-time Markov control processes}, Springer, New York, (1999).
  		
  			
  			

  			\bibitem{HEMA} D. HERNANDEZ-HERNANDEZ AND  S. I. MARCUS, \textit{Existence of risk-sensitive optimal stationary policies for controlled Markov processes}, Appl. Math. Optim. 40 (1999), pp. 273-285.
  			
		
  		
 		\bibitem{HM} R. A. HOWARD AND J. E. MATHESON, \textit{Risk-sensitive Markov decision processes}, Manag. Sci. 8 (1972), pp. 356-369.
 		
 			\bibitem{Jac} D. H. JACOBSON, \textit{Optimal stochastic linear systems with exponential performance criteria and their relation to stochastic  differential games}, IEEE Trans Automat Contr, 18 (1973), pp. 124-131.
%
%
 				\bibitem{DMS1} G. B. D. MASI AND L. STETTNER, \textit{Risk-sensitive control of discrete-time Markov processes with infinite horizon}, SIAM J. Control Optim. 38 (1999), pp. 61-78.
 			
 			\bibitem{DMS3} G. B. D. MASI AND L. STETTNER, \textit{Infinite horizon risk-sensitive control of discrete time Markov processes under minorization property}, SIAM J. Control Optim., 46 (2007), pp. 231-252.
 		

  	
  		
  			\bibitem{KP1} K.S. KUMAR AND C. PAL, \textit{Risk-sensitive control of jump process on denumerable state space with near monotone cost}, Appl. Math. Optim., 68 (2013), pp. 311-331.
  		
  			\bibitem{KP2} K.S. KUMAR AND C. PAL, \textit{Risk-sensitive control of continuous-time Markov processes with denumerable state space}, Stoch. Anal. Appl., 33 (2015), pp. 863-881.
  		
  		
  		

  		
  		 	\bibitem{WC6} Q. WEI AND X. CHEN, \textit{Risk-sensitive average continuous-time markov decision processes with unbounded rates}, Optimization, (2018).
  		 	
  			\bibitem{WC4} Q. D. WEI AND X. CHEN, \textit{Nonzero-sum games for continuous-time jupm processes under the expected Average payoff criterion},  Appl. Math. Optim., (2019).
  			
  				\bibitem{WC5} Q. D. WEI AND X. CHEN, \textit{Nonzero-sum Risk-Sensitive Average Stochastic Games: The Case of Unbounded Costs}, Dynamic games and Applications, 2021.
  				

  		\bibitem{WH} P. WHITTLE, \textit{Risk-sensitive linear quadratic Gaussian control}, Adv. Appl. Probab., 13 (1981), pp. 764-777.
  		
  			
  		%
  			\bibitem{Z1} W. Z. ZHANG, \textit{Average optimality for continuous-time Markov decision processes under weak continuty conditions}, J. of Appl. probab., 51 (2014), pp. 954-970.
  		
  			\end{thebibliography}
	\end{document}